%% file: asda_complete.tex
\documentclass[final]{siamltex1213}

\usepackage{algpseudocode} 
\usepackage{algorithmicx}
\usepackage{algorithm}

 \input{everything.tex}

\newcommand{\condnum}{\zeta}
\newcommand{\mb}{\tau} 
\newcommand{\sdim}{q}

\title{\bf Stochastic Reformulations of Linear Systems: \\  Algorithms and Convergence Theory
}

\author{Peter Richt\'{a}rik\thanks{King Abdullah University of Science and Technology (KAUST), Thuwal, Kingdom of Saudi Arabia, \email{peter.richtarik@kaust.edu.sa}} \and Martin Tak\'{a}\v{c}\thanks{Lehigh University, Bethlehem, Pennsylvania, USA, \email{mat614@lehigh.edu}}}

\begin{document}

\maketitle

\begin{abstract}
We develop a  family of reformulations of an arbitrary consistent linear system into  a {\em stochastic problem}. The reformulations are governed by two user-defined parameters: a positive definite matrix defining a norm, and an arbitrary discrete or continuous distribution over random matrices. Our reformulation has several equivalent interpretations, allowing for researchers from various communities to leverage their domain specific insights. In particular, our reformulation can be equivalently seen as a stochastic optimization problem, stochastic linear system, stochastic fixed point problem and a probabilistic intersection problem. We prove sufficient, and necessary and sufficient conditions for the reformulation to be exact.  Further, we propose and analyze three stochastic algorithms for solving the reformulated problem---basic, parallel and accelerated methods---with global linear convergence rates. The rates can be interpreted as  condition numbers of a matrix which depends on the system matrix and on the reformulation parameters. This gives rise to a new phenomenon  which we call {\em stochastic preconditioning}, and which refers to the problem of finding parameters (matrix and distribution) leading to a sufficiently small condition number. Our basic method can be equivalently interpreted as  stochastic gradient descent, stochastic 
Newton method, stochastic proximal point method, stochastic fixed point method, and stochastic projection method,  with fixed stepsize (relaxation parameter), applied to the reformulations. 
\end{abstract}



\section{Introduction} \label{sec:introduction}

Linear systems form the backbone of most numerical codes used in academia and industry. With the advent of the age of big data, practitioners are looking for ways to solve linear systems of unprecedented sizes. The present work is motivated by the need to design such algorithms.  As an algorithmic tool enabling faster computation, {\em randomization}  is well developed, understood and appreciated in several fields, typically traditionally of a ``discrete'' nature, most notably  theoretical computer science \cite{RA-book}. However,   probabilistic ideas are also increasingly and successfully penetrating ``continuous'' fields, such as numerical linear algebra \cite{Spielman2006, DKM:RandNLA1, DKM:RandNLA2, Strohmer2009, LSRN,  RandPrecond, Blendenpick}, optimization \cite{Leventhal:2008:RMLC, Lan2009, Nesterov:2010RCDM, UCDC, S.U.StichC.L.Muller2014,  ALPHA}, control theory \cite{Calafiore2000, Calafiore-Polyak2001, randControl},  machine learning \cite{SAG, PCDM, SVRG, APPROX, Hydra},  and signal processing \cite{Combettes-Pesquet2015, mS2GD}.

In this work\footnote{All theoretical results in this paper were obtained by August 2016 and a first draft was circulated to a few selected colleagues in September 2016. The first author gave several talks on these results before this draft was made publicly available on arXiv: {\em  Linear Algebra and Parallel Computing at the Heart of Scientific Computing}, Edinburgh, UK (Sept 21, 2016), {\em Seminar on Combinatorics, Games and Optimisation}, London School of Economics, London, UK, (Nov 16, 2016), {\em Workshop on Distributed Machine Learning}, T\'{e}l\'{e}com ParisTech, Paris, France (Nov 25, 2016),  Skoltech Seminar, Moscow, Russia (Dec 1, 2016), {\em BASP Frontiers Workshop}, Villars-sur-Ollon, Switzerland (Feb 1, 2017), and {\em SIAM Conference on Optimization}, Vancouver, Canada (May 22, 2017). In addition, the first author has included the results of this paper in the MSc/PhD course  {\em Modern optimization methods for big data problems}, delivered in Spring 2017 at the University of Edinburgh, as an introduction into the role of randomized decomposition in linear algebra, optimization and machine learning. All main results of this paper were distributed to the students in the form of slides.}  we are concerned with the problem of solving a consistent linear system.  In particular, consider the problem \begin{equation}\label{eq:linear_system}\text{solve} \quad \mA x = b,\end{equation}
where  $0\neq \mA\in \R^{m\times n}$.  We shall assume throughout the paper that the system is consistent, i.e., $\cL\eqdef \{x\;:\; \mA x = b\}\neq \emptyset$.  Problem \eqref{eq:linear_system} is arguably one of the  most important problems in linear algebra. As such, a tremendous amount of research  has been done to design efficient iterative algorithms  \cite{Saad:2003}. However, surprisingly little is know about randomized iterative algorithms for solving linear systems. In this work we aim to contribute to closing this gap.

\subsection{Stochastic reformulations of linear systems} \label{sec:intro:reform} We propose a  fundamental and flexible way of  reformulating each consistent linear system  into a {\em stochastic problem}. To the best of our knowledge, this is the first systematic study of such reformulations. Stochasticity is introduced in a controlled way,  into an otherwise deterministic problem, as a  decomposition tool which can be leveraged to design efficient, granular and scalable randomized algorithms.

\paragraph{Parameters defining the reformulation} Stochasticity enters our reformulations through a user-defined distribution $\cD$ describing an ensemble of random matrices $\mS\in \R^{m\times \sdim}$. We make use of one more parameter:  a user-defined $n\times n$ symmetric positive definite matrix $\mB$. Our approach and underlying theory support virtually all thinkable distributions\footnote{We only require that the the expectation $\EE{\bS \sim \cD}{\mH}$ exists, where $\mH\eqdef \mS(\mS^\top \mA \mB^{-1}\mA^\top \mS)^\dagger \mS^\top$.}. The choice of the distribution should ideally depend on the problem itself, as it will affect the conditioning of the reformulation. However, for now we leave such considerations aside. 

\paragraph{One stochastic reformulation in four disguises} Our reformulation of \eqref{eq:linear_system} as a stochastic problem has several seemingly different, yet equivalent interpretations, and hence we describe   them here side by side. 

 {\em a) Stochastic optimization problem.} Consider  the problem \begin{equation}\label{eq:problem:stoch_opt}\text{minimize} \quad f(x) \eqdef \EE{\mS\sim \cD}{f_{\mS}(x)},\end{equation}
where $f_{\mS}(x)=\tfrac{1}{2}(\mA x - b)^\top \mH (\mA x -b)$, $\mH = \mS(\mS^\top \mA \mB^{-1}\mA^\top \mS)^\dagger \mS^\top$, and $\dagger$ denotes the Moore-Penrose pseudoinverse.  When solving the problem, we do not have (or do not wish to exercise, as it may be prohibitively expensive) explicit access to $f$, its gradient or Hessian. Rather, we can repeatedly sample $\mS\sim \cD$ and receive unbiased samples of these quantities at points of interest. That is, we may obtain local information about the {\em stochastic function} $f_{\mS}$, such as the {\em stochastic gradient} $\nabla f_{\mS}(x)$, and use this to drive an iterative process for solving \eqref{eq:problem:stoch_opt}. 

 {\em b) Stochastic linear system.} Consider now a preconditioned version of the linear system \eqref{eq:linear_system} given by
\begin{equation}\label{eq:problem:preconditioning}\text{solve} \quad \mB^{-1}\mA^\top \EE{\mS \sim \cD}{\mH} \mA x = \mB^{-1}\mA^\top \EE{\mS\sim \cD}{\mH} b,\end{equation}
where $\mP = \mB^{-1}\mA^\top \EE{\mS \sim \cD}{\mH}$ is the preconditioner.  The preconditioner is not assumed to be known explicitly. Instead, when solving the problem, we are able to repeatedly sample $\mS\sim \cD$, obtaining an unbiased estimate of the preconditioner (not necessarily explicitly), $\mB^{-1}\mA^\top \mH$, for which we coin the name  {\em stochastic preconditioner}. This gives us access to an unbiased sample of the preconditioned system \eqref{eq:problem:preconditioning}: $\mB^{-1}\mA^\top \mH \mA x = \mB^{-1}\mA^\top \mH b.$ As we shall see---in an analogy with stochastic optimization---the   information contained in such systems can be utilized by  an iterative algorithm to solve  \eqref{eq:problem:preconditioning}.

 {\em c) Stochastic fixed point problem.} Let $\Pi^{\mB}_{\cL_{\mS}}(x)$ denote the projection of $x$ onto $\cL_{\mS}\eqdef \{x\;:\; \mS^\top \mA x = \mS^\top b\}$, in the norm $\|x\|_{\mB} \eqdef \sqrt{x^\top \mB x}$. Consider the {\em stochastic fixed point problem} 
\begin{equation}\label{eq:problem:fixed_point} \text{solve} \quad x = \EE{\mS\sim \cD}{\Pi^{\mB}_{\cL_{\mS}}(x)}.\end{equation}
That is, we seek to find a  {\em fixed   point} of the mapping $x\to \EE{\mS\sim \cD}{\Pi^{\mB}_{\cL_{\mS}}(x)}$.  When solving the problem, we do not have an explicit access to the average projection map. Instead, we are able to repeatedly sample $\mS\sim \cD$, and use the stochastic projection map $x\to \Pi^{\mB}_{\cL_{\mS}}(x)$.

 {\em d) Probabilistic intersection problem.} Note that $\cL\subseteq \cL_{\mS}$ for all  $\mS$. We would wish to design $\cD$ in such a way that a suitably chosen notion of an intersection of the sets $\cL_{\mS}$ is equal to $\cL$. The correct notion is what we call {\em probabilistic intersection}, denoted $\cap_{\mS\sim \cD} \cL_{\mS}$, and defined as the set of points $x$ which belong to $\cL_{\mS}$ with probability one. This leads to the problem:
\begin{equation}\label{eq:problem:intersection}\text{find} \quad  x \in \cap_{\mS\sim \cD} \cL_{\mS} \eqdef \{x\;:\; \Prob(x\in \cL_{\mS}) =1 \}.\end{equation}

As before, we typically do not have an explicit access to the probabilistic intersection when designing an algorithm. Instead, we can repeatedly sample $\mS\sim \cD$, and utilize the knowledge of $\cL_{\mS}$ to drive the iterative process. If $\cD$ is a discrete distribution, probabilistic intersection reduces to standard intersection.

All of the above formulations have a common feature: they all involve an expectation over $\mS\sim \cD$, and we either do not assume this expectation is known explicitly, or even if it is, we prefer, due to efficiency or other considerations, to sample from unbiased estimates of the objects (e.g., stochastic gradient $\nabla f_{\mS}$, stochastic preconditioner $\mB^{-1}\mA^\top \mH$, stochastic projection map $x\to \Pi^{\mB}_{\cL_{\mS}}(x)$, random set $\cL_{\mS}$) appearing in the formulation. 

\paragraph{Equivalence and exactness} We show that all these stochastic  reformulations are equivalent (see Theorem~\ref{thm:X}). In particular, the following sets are identical: the set of minimizers of the stochastic optimization problem \eqref{eq:problem:stoch_opt}, the solution set of the stochastic linear system \eqref{eq:problem:preconditioning},  the set of fixed points of the stochastic fixed point problem \eqref{eq:problem:fixed_point},  and the probabilistic intersection \eqref{eq:problem:intersection}.  Further, we give necessary and sufficient  conditions for this set to be equal to  $\cL$. If this is the case, we say the the reformulation is {\em exact} (see Section~\ref{sec:exactness}).  Distributions $\cD$ satisfying these conditions always exist, independently of any assumptions on the system beyond consistency. The simplest, but also the least useful choice of a distribution is to pick $\mS = \mI$ (the $m\times m$ identity matrix), with probability one.  In this case, all of our reformulations become trivial.

\subsection{Stochastic algorithms} \label{sec:intro_algs}

Besides proposing a family of stochastic reformulations of the linear system \eqref{eq:linear_system}, we also propose three stochastic algorithms for solving them: Algorithm~\ref{alg:alg1} (basic method), Algorithm~\ref{alg:alg2} (parallel/minibatch method), and Algorithm~\ref{alg:alg3} (accelerated method). Each method can be interpreted naturally from the viewpoint of each of the reformulations. 

\paragraph{Basic method} Below we list some of the interpretations of Algorithm~\ref{alg:alg1} (basic method), which performs updates of the form
\begin{equation} \label{eq:alg1-informal} x_{k+1} = \phi_{\omega}(x_k,\mS_k) \eqdef x_k - \omega \mB^{-1} \mA^\top  \mS_k (\mS_k^\top  \mA \mB^{-1} \mA^\top  \mS_k)^\dagger \mS_k^\top  (\mA x_k - b), \end{equation}
where $\mS_k\sim \cD$ is sampled independently in each iteration.  The method is formally presented and analyzed in Section~\ref{sec:Basic_Method}.

 {\bf a) Stochastic gradient descent.} Algorithm~\ref{alg:alg1} can be seen as {\em stochastic gradient descent} \cite{RobbinsMonro:1951}, with fixed stepsize, applied to \eqref{eq:problem:stoch_opt}. At  iteration $k$ of the method, we sample $\mS_k\sim \cD$, and compute $\nabla f_{\mS_k}(x_k)$, which is an unbiased stochastic approximation of $\nabla f(x_k)$. We then perform the step \begin{equation}\label{eq:alg:SGD}x_{k+1}=x_k - \omega \nabla f_{\mS_k}(x_k),\end{equation} where $\omega>0$ is a stepsize. 

Let us note that in order to achieve linear convergence it is not necessary to use any explicit variance reduction strategy \cite{SAG, SVRG, S2GD, SAGA}, nor do we need to use decreasing stepsizes. This is because the stochastic gradients vanish at the optimum, which is a consequence of the consistency assumption. Surprisingly, we get linear convergence in spite of the fact that we deal with a  non-finite-sum problem \eqref{eq:problem:stoch_opt}, and without the need to assume boundedness of the stochastic gradients, and without $f$ being strongly convex. To the best of our knowledge, this is the first linearly convergent accelerated method for stochastic optimization without requiring strong convexity. This beats the minimax bounds given by Srebro~\cite{SrebroSGDconvex2010}. This is because \eqref{eq:problem:stoch_opt} is not a black-box stochastic optimization objective; indeed, we have constructed it in a particular way from the underlying linear system \eqref{eq:linear_system}. 


 {\bf  b) Stochastic Newton method.} However, Algorithm~\ref{alg:alg1} can also be seen as a {\em stochastic Newton method}. At iteration $k$ we sample $\mS_k\sim \cD$, and instead of applying the inverted  Hessian of $f_{\mS_k}$ to the stochastic gradient (this is not possible as the Hessian is not necessarily invertible), we apply a pseudoinverse (which always exists).  That is, we perform the step \begin{equation}\label{eq:alg:SNM}x_{k+1} = x_k - \omega (\nabla^2 f_{\mS_k}(x_k))^{\dagger_{\mB}} \nabla f_{\mS_k}(x_k),\end{equation}
where $\omega>0$ is a stepsize and the $\mB$-pseudoinverse of a matrix $\mM$ is defined as
$\mM^{\dagger_\mB} \eqdef \mB^{-1} \mM^\top (\mM \mB^{-1}\mM^\top)^\dagger$. While {\em any} pseudoinverse will resolve non-invertibility issue, since we work in the geometry induced by the $\mB$--inner product,  the $\mB$-pseudoinverse is the right choice.  One may wonder why methods \eqref{eq:alg:SGD} and \eqref{eq:alg:SNM} are equivalent; after all,  the (stochastic) gradient descent and (stochastic) Newton methods are not equivalent in general. However, in our setting it turns out that the stochastic gradient $\nabla f_{\mS_k}(x)$ is always an eigenvector of $(\nabla^2 f_{\mS_k}(x))^{\dagger_{\mB}}$,  with eigenvalue 1 (see Lemma~\ref{lem:all_sort_of_stuff_is_equal}).  

Stochastic Newton-type methods were recently developed and analyzed in the optimization and machine learning literature \cite{HessianSketch, NewtonSketch, SDNA, PSNM}. However, they are designed to solve different problems, and operate in a different manner.

 {\bf c) Stochastic proximal point method.} If we restrict our attention to stepsizes satisfying $0<\omega \leq 1$, then Algorithm~\ref{alg:alg1} can be equivalently (see Theorem~\ref{thm:SPP} in the Appendix) written down as 
\begin{equation} 
\label{eq:alg:SPP} x_{k+1} = \arg \min_{x\in \R^n} \left\{f_{\mS_k}(x) + \frac{1-\omega}{2\omega}\|x-x_k\|_{\mB}^2\right\}.
\end{equation}
That is, \eqref{eq:alg:SPP} is a {\em stochastic} variant of the {\em proximal point method}  for solving \eqref{eq:problem:stoch_opt}, with a fixed regularization parameter \cite{PPM1976}. The proximal point method is obtained from \eqref{eq:alg:SPP} by replacing  $f_{\mS_k}$ with $f$.  If we define the {\em prox operator} of a function $\psi:\R^n\to \R$ with respect to the $\mB$-norm as
$\prox_{\psi}^{\mB}(y) \eqdef \arg \min_{x\in \R^n} \left\{\psi(x) + \frac{1}{2}\|x-y\|_{\mB}^2\right\},$
 then iteration \eqref{eq:alg:SPP} can be written compactly as
$x_{k+1}  = \prox^{\mB}_{\tfrac{\omega}{1-\omega} f_{\mS_k}}(x_k).$

{\bf d) Stochastic fixed point method.} From the perspective of the stochastic fixed point problem \eqref{eq:problem:fixed_point},  Algorithm~\ref{alg:alg1}  can be interpreted as a {\em stochastic fixed point method}, with relaxation. We first reformulate the problem into an equivalent form using relaxation, which is done to improve the contraction properties of the map. We pick a parameter $\omega>0$, and instead consider the equivalent fixed point problem
  $x = \EE{\mS\sim \cD}{\omega \Pi^{\mB}_{\cL_{\mS}}(x) + (1-\omega)x}$. Now, at iteration $k$, we sample $\mS_k\sim \cD$, which enables us to obtain an unbiased estimate of the new fixed point mapping, and then simply perform one step of a fixed point method on this mapping:
\begin{equation}\label{eq:alg:SFP} x_{k+1} = \omega \Pi^{\mB}_{\cL_{\mS_k}}(x_k) + (1-\omega)x_k.\end{equation}

{\bf e) Stochastic projection method.} Algorithm~\ref{alg:alg1} can also be seen as a {\em stochastic projection method} applied to the probabilistic intersection problem \eqref{eq:problem:intersection}. By sampling $\mS_k\sim \cD$, we are one of the sets defining the intersection, namely  $\cL_{\mS_k}$. We then project the last iterate onto this set, in the $\mB$-norm, followed by a relaxation step with relaxation parameter $\omega>0$. That is, we perform the update \begin{equation}\label{eq:alg:SPM} x_{k+1} = x_k + \omega (\Pi^{\mB}_{\cL_{\mS_k}}(x_k)-x_k).\end{equation}
This is a randomized variant of an alternating projection method. Note that the representation of $\cL$ as a probabilistic intersection of sets is not given to us. Rather,  we  construct it with the hope to obtain faster convergence. 

An optimization algorithm utilizing stochastic projection steps was developed in \cite{randomProj-Nedic2011}. For a comprehensive  survey of projection methods for convex feasibility problems, see \cite{BauBor:96}.

\paragraph{Parallel method} A natural parallelization strategy is to perform one step of the basic method independently $\mb$ times,  starting from the same point $x_k$, and average the results:
\begin{equation} \label{eq:parallel_method_intro}
x_{k+1} = \frac{1}{\mb}\sum_{i=1}^\mb \phi_{\omega}(x_k,\mS_k^i),
\end{equation}
where $\mS_k^1,\dots,\mS_k^\mb$ are independent samples from $\cD$ (recall that $\phi_\omega$ is defined in \eqref{eq:alg1-informal}). This method is formalized as Algorithm~\ref{alg:alg2}, and studied in Section~\ref{sec:minibatch}.  Betrayed by our choice of the name, this method is useful in scenarios where $\mb$ parallel workers are available, allowing for the $\mb$ basic steps to be computed in parallel, followed by an averaging operation. 

From the stochastic optimization viewpoint, this is a {\em minibatch} method. Considering the SGD interpretation \eqref{eq:alg:SGD},  we can equivalently write \eqref{eq:parallel_method_intro} in the form
$x_{k+1} = x_k - \frac{1}{\mb}\sum_{i=1}^\mb \nabla f_{\mS_k^i}(x_k).$ This is {\em minibatch SGD}. Iteration complexity of minibatch SGD was first understood in the context of training support vector machines with the hinge loss \cite{pegasos2}.  Complexity under a lock-free paradigm, in a different setting from ours, was first studied in \cite{hogwild}. Notice that in the limit $\mb\to \infty$, we obtain gradient descent. It is therefore interesting to study the complexity of the parallel method as a function $\mb$.  Of course, this method can also be interpreted as a minibatch stochastic Newton method, minibatch proximal point method and so on. 

From the probabilistic intersection point of view, method \eqref{eq:parallel_method_intro} can be interpreted as a stochastic variant of the parallel projection method. In particular, we obtain the iterative process
\[x_{k+1} = x_k + \omega \left[\left(\frac{1}{\mb}\sum_{i=1}^\mb \Pi^{\mB}_{\cL_{\mS_k^i}}(x_k) \right)-x_k\right].\]
That is, we move from the current iterate, $x_k$, towards the average of the $\mb$ projection points, with undershooting (if $\omega<1$), precisely landing on (if $\omega=1$), or overshooting (if $\omega>1$) the average. Projection methods have a long history and are well studied \cite{Dykstra1983,Boyle_and_Dykstra1986}. However, much less is known about stochastic projection methods.

\paragraph{Accelerated method} In order to obtain  acceleration without parallelization---that is, acceleration in the sense of Nesterov~\cite{NesterovBook}---we suggest to perform an update step in which $x_{k+1}$ depends on both $x_{k}$ and $x_{k-1}$. In particular, we take two {\em dependent} steps of Algorithm~\ref{alg:alg1}, one from $x_k$ and one from $x_{k-1}$, and then take an affine combination of the results. That is, the process is started with $x_0,x_1\in \R^n$, and for $k\geq 1$ involves an iteration of the form
\begin{equation} \label{eq:alg2-intro} x_{k+1} = \gamma \phi_{\omega}(x_k,\mS_k) + (1-\gamma)\phi_{\omega}(x_{k-1},\mS_{k-1}),\end{equation}
where  the  matrices $\{\mS_k\}$ are  independent samples from $\cD$, and $\gamma\in \R$ is an {\em acceleration parameter}. Note that by choosing $\gamma=1$ (no acceleration), we recover Algorithm~\ref{alg:alg1}. This method is formalized as Algorithm~\ref{alg:alg3} and analyzed in Section~\ref{sec:acceleration}. Our theory suggests that $\gamma$ should be always between $1$ and $2$. In particular, for well conditioned problems\footnote{The condition number, $\condnum$, is defined in  \eqref{eq:condition_number}.}, one should choose $\gamma\approx 1$, and for ill conditioned problems, one should choose $\gamma\approx 2$.

\subsection{Complexity}

The iteration complexity of our methods is completely described  by the spectrum of the (symmetric positive semidefinite) matrix
\[\mW \eqdef \mB^{-1/2}\mA^\top \EE{\mS\sim \cD}{\mH} \mA \mB^{-1/2}.\] 
Let $\mW = \mU \Lambda \mU^\top $ be the eigenvalue decomposition of $\mW$, where $\mU=[u_1,\dots,u_n]$ are the eigenvectors,  $\lambda_1\geq \lambda_2 \geq \dots \geq \lambda_n \geq 0$ are the eigenvalues, and $\Lambda = \Diag{\lambda_1,\dots,\lambda_n}$. It can be shown that the largest eigenvalue, $\lambda_{\max} \eqdef \lambda_1 $ is bounded above by 1 (see Lemma~\ref{lem:spectrumxx}).  Let $\lambda_{\min}^+$ be the smallest nonzero eigenvalue. 

With all of the above reformulations we associate the same {\em condition number}
\begin{equation} \label{eq:condition_number}\condnum = \condnum(\mA, \mB, \cD) \eqdef \|\mW\| \|\mW^{\dagger}\| = \frac{\lambda_{\max}}{\lambda_{\min}^+},\end{equation}
where $\|\cdot\|$ is the spectral norm, $\lambda_{\max}$ is the largest eigenvalue of $\mW$ and $\lambda_{\min}^+$ is the smallest nonzero eigenvalue of $\mW$. Note that, for example, $\condnum$ is the condition number of the Hessian of $f$, and also the condition number of the stochastic linear system \eqref{eq:problem:preconditioning}. Natural interpretations from the viewpoint of the stochastic fixed point and probabilistic intersection problems are also possible. As one varies the parameters defining the reformulation ($\cD$ and $\mB$), the condition number changes.  For instance, choosing $\mS=\mI$ with probability one gives $\condnum=1$.

\begin{table}
\begin{center}
{
\scriptsize
\begin{tabular}{|c|c|c|c|c|c|c|c|}
\hline
Alg. & $\omega$ & $\mb$ & $\gamma$& Quantity & Rate & Complexity & Theorem\\
\hline

\ref{alg:alg1} & $1$ & - & - &  $\|\Exp{x_k-x_*}\|_{\mB}^2$ & $(1-\lambda_{\min}^+)^{2k}$        & $1/\lambda_{\min}^+$ & \ref{thm:alg1_complexity_first}, \ref{thm:corollary}, \ref{thm:alg1_bestomega} \\

\ref{alg:alg1} & $1/\lambda_{\max}$ & - & - &  $\|\Exp{x_k-x_*}\|_{\mB}^2$ & $(1-1/\condnum)^{2k}$        & $\condnum$ & \ref{thm:alg1_complexity_first}, \ref{thm:corollary}, \ref{thm:alg1_bestomega} \\

\ref{alg:alg1} & $\tfrac{2}{\lambda_{\min}^+ + \lambda_{\max}}$ & - & - &  $\|\Exp{x_k-x_*}\|_{\mB}^2$ & $(1-2/(\condnum+1))^{2k}$        & $\condnum$ & \ref{thm:alg1_complexity_first}, \ref{thm:corollary}, \ref{thm:alg1_bestomega} \\

\ref{alg:alg1} & 1 & -& -& $\Exp{\|x_k-x_*\|_{\mB}^2}$ & $(1-\lambda_{\min}^+)^k$        &  $1/\lambda_{\min}^+$ & \ref{thm:alg1_complexity_second}\\

\ref{alg:alg1} & 1 & - & - &  $\Exp{f(x_k)}$ & $(1-\lambda_{\min}^+)^k$        & $1/\lambda_{\min}^+$ & \ref{thm:complexity_f_2}\\

\hline

\ref{alg:alg2} & $1$ & $\mb$ & - & $\Exp{\|x_k-x_*\|_{\mB}^2}$ & $\left( 1-\lambda_{\min}^+ \left(2- \xi(\mb)\right) \right)^k $ &  &  \ref{thm:minibatch} \\

\ref{alg:alg2} & $1/\xi(\mb)$ & $\mb$ & - & $\Exp{\|x_k-x_*\|_{\mB}^2}$ & $\left( 1-\tfrac{\lambda_{\min}^+}{\xi(\mb) }\right)^k $ & $\xi(\mb)/\lambda_{\min}^+ $ &  \ref{thm:minibatch} \\

\ref{alg:alg2} & $1/\lambda_{\max}$ & $\infty$ & - & $\Exp{\|x_k-x_*\|_{\mB}^2}$ & $(1-1/\condnum)^k$ & $\condnum$ &  \ref{thm:minibatch} \\

\hline

\ref{alg:alg3} & $1$& - & $\tfrac{2}{1+\sqrt{0.99\lambda_{\min}^+}}$ &  $\|\Exp{x_k-x_*}\|_{\mB}^2$ & $\left(1-\sqrt{0.99\lambda_{\min}^+}\right)^{2k}$ & $\sqrt{1/\lambda_{\min}^+}$ & \ref{thm:main-accelerated}\\

\ref{alg:alg3} & $1/\lambda_{\max}$& - & $\tfrac{2}{1+\sqrt{0.99/\condnum}}$ &  $\|\Exp{x_k-x_*}\|_{\mB}^2$ & $\left(1-\sqrt{0.99/\condnum}\right)^{2k}$ & $\sqrt{\condnum}$ & \ref{thm:main-accelerated}\\

\hline
\end{tabular}
}
\end{center}
\caption{Summary of the main complexity results. In all cases, $x_*=\Pi^{\mB}_{\cL}(x_0)$ (the projection of the starting point onto the solution space of the linear system). ``Complexity'' refers to the number of iterations needed to drive ``Quantity'' below some error tolerance $\epsilon>0$ (we suppress a $\log(1/\epsilon)$ factor in all expressions in the ``Complexity'' column).  In the table we use the following expressions: $\xi(\mb) = \tfrac{1}{\mb} + (1-\tfrac{1}{\mb}) \lambda_{\max}$ and $\condnum = \lambda_{\max}/\lambda_{\min}^+$.}
\label{tbl:summary_intro}
\end{table}

\paragraph{Exact formula for the evolution of expected iterates}
We first show (Theorem~\ref{thm:alg1_complexity_first}) that after the canonical linear transformation $x\mapsto \mU^\top \mB^{1/2}x$, the expected iterates of the basic method satisfy the identity
\begin{equation}\label{eq:i98t8ghdiee}\Exp{\mU^\top  \mB^{1/2}(x_k-x_*)} = (\mI -  \omega  \Lambda )^k \mU^\top  \mB^{1/2}(x_0-x_*),\end{equation}
where $x_*$ is an arbitrary solution of the linear system (i.e., $x_*\in \cL$). This identity seems to suggest that zero eigenvalues cause an issue, preventing convergence of the corresponding elements of the error to zero. Indeed, if $\lambda_i=0$, then 
\eqref{eq:i98t8ghdiee} implies that $u_i^\top \mB^{1/2}(x_k-x_*) = u_i^\top \mB^{1/2}(x_0-x_*)$, which does not change with $k$. However, it turns out that under the assumption of exactness we have $ u_i^\top \mB^{1/2}(x_0-x_*)=0$ whenever $\lambda_i=0$ if we let $x_*$ to be the projection, in the $\mB$-norm, of $x_0$ onto $\cL$ (Theorem~\ref{thm:alg1_complexity_first}). This is then used to argue (Corollary~\ref{thm:corollary})  that $\|\Exp{x_k-x_*}\|_\mB$ converges to zero if and only if $0<\omega<2/\lambda_{\max}$. The choice of stepsize issue is discussed in detail in Section~\ref{subsec: omega}.

The main complexity results obtained in this paper are summarized in Table~\ref{tbl:summary_intro}. The full statements including the dependence of the rate on these parameters, as well as other alternative results (such as lower bounds, ergodic convergence)  can be found in the theorems referenced in the table.

\paragraph{L2 (mean square) convergence}The rate of decay of the quantity $\|\Exp{x_k-x_*}\|_\mB^2$ for three different stepsize choices is summarized in the first three rows  of Table~\ref{tbl:summary_intro}. In particular, the default stepsize $\omega=1$  leads to the complexity $1/\lambda_{\min}^+$, the long stepsize $\omega = 1/\lambda_{\max}$  gives the improved complexity  $\lambda_{\max}/\lambda_{\min}^+$, and 
the optimal stepsize $\omega = 2/(\lambda_{\max}+\lambda_{\min}^+)$  gives the  best complexity $0.5+0.5\lambda_{\max}/\lambda_{\min}^+$. However, if we are interested in the convergence of the larger quantity $\Exp{\|x_k-x_*\|_\mB^2}$ (L2 convergence), it turns out that $\omega=1$ is the optimal choice, leading to the complexity $1/\lambda_{\min}^+$.

\paragraph{Parallel and accelerated methods} The parallel method improves upon the basic method in that it is capable of faster L2 convergence. We give a complexity formula as a function of  $\tau$, recovering the complexity the $1/\lambda_{\min}^+$ rate of the basic method in the $\tau=1$ case, and achieving the improved asymptotic complexity $\lambda_{\max}/\lambda_{\min}^+$ as $\tau\to \infty$ (recall that $\lambda_{\max}\leq 1$, whence the improvement). Because of this,  $\lambda_{\max}$ is the quantity driving {\em parallelizability}. If $\lambda_{\max}$ is close to 1, then there is little or no reason to parallelize. If $\lambda_{\max}$ is very small, parallelization helps. The smaller $\lambda_{\max}$ is, the more gain is achieved by utilizing more processors.

With an appropriate choice of the stepsize ($\omega$)  and acceleration ($\gamma$) parameters, the accelerated method improves the complexity $\lambda_{\max}/\lambda_{\min}^+$ achieved by the basic method to $\sqrt{\lambda_{\max}/\lambda_{\min}^+}$. However, this is established for the quantity $\|\Exp{x_k-x_*}\|_\mB^2$. We conjecture that the L2 convergence rate of the accelerated method (for a suitable choice of the parameters $\omega$ and $\gamma$) is $\sqrt{1/\lambda_{\min}^+}$. 

\paragraph{Novelty} All convergence results presented in this paper are new in one way or another. Indeed, we extend the methods from \cite{SIMAX2015, SDA} to include a stepsize/relaxation parameter, or analyze these methods  under a weaker condition (basic method with $\omega=1$ under the ``exactness'' assumption), or develop new methods (parallel and accelerated variants).  Let us focus on the case of the basic method with unit stepsize as this method was considered in \cite{SIMAX2015}  and \cite{SDA}: see lines 1 and 4 of Table~\ref{tbl:summary_intro}. Even in this case, our results hold under weaker assumptions; and the unit stepsize is not optimal for the convergence rate of the quantity in line 1 (this can be seen by looking at the improved rates in lines 2 and 3). While \cite{SDA} weakens the (rather strong) assumption in \cite{SIMAX2015} (full column rank of $\bf A$), our assumption (which we call ``exactness'') is even weaker. The focus of paper \cite{SDA} was both to weaken the assumptions in \cite{SIMAX2015}, and also to develop a duality theory for what we now call the ``basic method''. Here we do not touch on duality theory at all (the extension is rather straightforward).


\subsection{Stochastic preconditioning} We coin the phrase {\em stochastic preconditioning} to refer to the general problem of {\em designing} matrix $\mB$ and distribution $\cD$ such that
the appropriate condition number of $\mW$ is well behaved. For instance, one might be interested in minimizing (or reducing) the condition number $1/\lambda_{\min}^+$ if the basic method with unit stepsize is used, and the quantity we wish to converge to zero is either $\Exp{\|x_k-x_*\|_{\mB}^2}$, $\|\Exp{x_k-x_*}\|_{\mB}^2$, or $\Exp{f(x_k)}$ (see Lines 1, 4 and 5 of Table~\ref{tbl:summary_intro}). On the other hand, if we can estimate $\lambda_{\max}$, then one may use the basic method with the larger stepsize $1/\lambda_{\max}$, in which case we may wish to minimize (or reduce) the condition number $\lambda_{\max}/\lambda_{\min}^+$ (see Line 2 of Table~\ref{tbl:summary_intro}).

One possible approach to stochastic preconditioning is to choose some $\mB$ and then focus on a reasonably simple parametric family of distributions $\cD$, trying to find the parameters which minimize (or reduce) the condition number of interest. The distributions in this family should be ``comparable'' in the sense that the cost of one iteration of the method of interest should be comparable for all distributions; as otherwise comparing bounds on the number of iterations does not make sense.

To illustrate this through a simple example,  consider the family of discrete uniform distributions over $m$ vectors in $\R^m$ (that is, $\mS_1, \dots, \mS_m \in \R^{m\times 1}$), where the vectors themselves are the parameters defining the family. The cost of one iteration of the basic method will be proportional to the cost of performing a matrix-vector product of the form $\mS^\top \mA$, which is comparable across all distributions in the family (assuming the vectors are dense, say). To illustrate this approach, consider this family, and further assume that $\mA$ is $n\times n$ symmetric and positive definite. Choose $\mB=\mA$. It can be shown that $1/\lambda_{\min}^+$  is maximized precisely when $\{\mS_i\}$ correspond to the eigenvectors of $\mA$.  In this case, $1/\lambda_{\min}^+ = n$, and hence our stochastic preconditioning strategy results in a condition number which is {\em independent} of the condition number of $\mA$. If we now apply the basic method to the stochastic optimization reformulation, we can interpret it as a {\em spectral} variant of stochastic gradient descent (spectral SGD). Ignoring logarithmic terms, spectral SGD only needs to perform $n$ matrix vector multiplications to solve the problem. While this is not a practical preconditioning strategy---computing the eigenvectors is hard, and if we actually had access to them, we could construct the solution directly, without the need to resort to an iterative scheme---it sheds light on the opportunities and challenges associated with stochastic preconditioning.

All  standard sketching matrices $\bS$ can be employed within our framework, including the count sketch~\cite{CountSketch2002} and the count-min sketch \cite{CountMinSketch2005}. In the context to this paper (since we sketch with the transpose of $\mS$), $\mS$ is a count-sketch matrix (resp. count-min sketch) if it is assembled from random columns of $[\mI,-\mI]$ (resp $\mI$), chosen uniformly with replacement, where $\mI$ is the $m\times m$ identity matrix.

 The notion of {\em importance sampling} developed in the last 5 years in the randomized optimization and machine learning literature \cite{NSync, IProx-SDCA, QUARTZ, ALPHA} can be seen a type of stochastic preconditioning, somewhat reverse to what we have outlined above. In these methods, the atoms forming the distribution $\cD$ are fixed, and one is seeking to associate them with appropriate probabilities. Thus, the probability simplex is the parameter space defining the class of distributions one is considering.
 
Stochastic preconditioning is  fundamentally different from the idea of {\em randomized preconditioning} \cite{RandPrecond, Blendenpick}, which is based on a two-stage procedure. In the first step, the input matrix is randomly projected and an good preconditioning matrix is extracted. In the second step, an iterative least squares solver  is applied to solve the preconditioned system. 
 
Much like standard preconditioning, different stochastic preconditioning strategies will need to be developed for different classes of problems, with structure of $\mA$ informing the choice of $\mB$ and $\cD$. Due to its inherent difficulty, stochastic preconditioning is beyond  the scope of this paper.


\subsection{Notation}

For convenience, a table of the most frequently used notation is included in Appendix~\ref{sec:notation_glossary}.  All matrices are written in bold capital letters. By $\range{\mM}$ and $\kernel{\mM}$ we mean the range space and null space of matrix $\mM$, respectively.  Given a symmetric positive definite matrix $\mB \in \R^{n\times n} $, we equip $\R^n$ with the Euclidean inner
product defined by $\dotprod{x, h}_{\mB} \eqdef x^\top \mB h$. We also define the induced norm: $\norm{x}_{\mB} \eqdef \sqrt{\dotprod{x, x}_{\mB}}$.  The short-hand notation $\|\cdot\|$ means $\|\cdot\|_{\mI}$, where $\mI$ is the identity matrix. We shall often write $\|x\|_{\mM}$ for matrix $\mM\in \R^{n\times n}$ being merely positive semidefinite; this constitutes a pseudonorm.

\section{Further Connections to Existing Work}

In this section we outline several connections of our work with existing  developments. We do not aim to be comprehensive.

\subsection{Randomized Kaczmarz method, with relaxation and acceleration}

Let $\mB = \mI$, and choose $\cD$ as follows: $\mS = e_i$ with probability $p_i = \norm{\mA_{i:}}_2^2 / \norm{\mA}_F^2$.  Since
\[\mW = \mB^{-1/2} \Exp{\mZ} \mB^{-1/2} = \Exp{\mZ} =
\sum_{i=1}^m p_i
\frac{\mA_{i:}^\top \mA_{i:}}{\norm{\mA_{i:}}_2^2}
=
\frac{1}{\norm{\mA}_F^2} \sum_{i=1}^m  \mA_{i:}^\top \mA_{i:} 
= \frac{\mA^\top \mA}{\norm{\mA}_F^2}.
\] 
the condition number is
\begin{equation}
\label{eq:iuiughsss}\condnum = \|\mW\| \|\mW^{\dagger} \| = \|\Exp{\mZ}\| \|\Exp{\mZ}^{\dagger}\| = \frac{\lambda_{\max}(\mA^\top \mA)}{\lambda_{\min}^+(\mA^\top \mA)}.\end{equation}

\paragraph{Basic method} In this setup,  Algorithm~\ref{alg:alg1} simplifies to
\[x_{k+1} = x_k - \frac{\omega  (\mA_{i:}x_k - b_i)}{\|\mA_{i:}\|_2^2}\mA_{i:}^\top.\]

For $\omega=1$, this reduces to the celebrated randomized Kaczmarz method (RK) of Strohmer and Vershyinin \cite{Strohmer2009}. For $\omega>1$, this is {\em RK with overrelaxation} -- a new method  not considered before. Based on Theorem~\ref{thm:alg1_bestomega}, for $\omega \in [1/\lambda_{\max}, \omega_*]$ the iteration complexity of Algorithm~\ref{alg:alg1} is  ${\tilde \cO}(\condnum)  \overset{\eqref{eq:iuiughsss}}{=} {\tilde \cO}\left(\frac{\lambda_{\max}(\mA^\top \mA)}{\lambda_{\min}^+(\mA^\top \mA)}\right).$ This is an improvement on standard RK method (with $\omega=1$), whose complexity  depends on $\trace{\mA^\top \mA}$ instead of $\lambda_{\max}$.  Thus, the improvement  can be as large as by a factor $n$.

\paragraph{Accelerated method}  In the same setup, Algorithm~\ref{alg:alg3} simplifies to
\[x_{k+1} = \gamma\left(x_k - \tfrac{\omega (\mA_{i_k :}x_k - b_{i_k})}{\|\mA_{i_k :}\|_2^2}\mA_{i_k :}^\top \right) + (1-\gamma)\left(x_{k-1} - \tfrac{ \omega (\mA_{i_{k-1} :}x_{k-1} - b_{i_{k-1}})}{\|\mA_{i_{k-1} :}\|_2^2}\mA_{i_{k-1} :}^\top\right)\]
This is accelerated RK method with overrelaxation -- a new method not considered before.  Based on Theorem~\ref{thm:main-accelerated}, for the parameter choice $\omega = 1/\lambda_{\max}$ and $\gamma = 2/(1+\condnum^{-2})$, the iteration complexity of this method is ${\tilde \cO}(\sqrt{\condnum}) \overset{\eqref{eq:iuiughsss}}{=} {\tilde \cO} \left( \sqrt{\frac{\lambda_{\max}(\mA^\top \mA)}{\lambda_{\min}^+(\mA^\top \mA)}} \right).$
If we instead choose $\omega=1$ and $\gamma = 2/(1+\condnum^{-2})$, the iteration complexity gets slightly worse: $1/\sqrt{\lambda_{\min}^+(\mW)} = \|\mA\|_F/ \sqrt{\lambda_{\min}^+(\mA^\top \mA)}$. To the best of our knowledge, this is the best known complexity for a variant of RK.
Let us remark that an asynchronous accelerated RK method was developed in \cite{liu2015accelerated}.

The randomized Kaczmarz method, its variants have received considerable attention recently  \cite{Needell2010,  Zouzias2012,  Needell2012a, Ramdas2014}, and several connections to existing methods were made. Kaczmarz-type methods in a Hilbert setting were developed in \cite{Oswald2015}.

\subsection{Basic method with unit stepsize}

The method $x_{k+1}\leftarrow \Pi_{\cL_{\mS_k}}^\mB(x_k)$  was first proposed and analyzed (under a full rank assumption on $\mA$)  in \cite{SIMAX2015}. Note that in view of \eqref{eq:alg:SFP}, this is the basic method with unit stepsize. However, it was not interpreted as a method for solving any of the reformulations presented here, and as a result, all the interpretations we are giving here also remained undiscovered. Instead, it was developed and presented as a method for finding the unique solution of~\eqref{eq:linear_system}.

\subsection{Duality}

As we have seen, all three methods developed in this paper converge to a specific solution of the linear system \eqref{eq:linear_system}, namely, to the projection of the starting point onto the solution space: $x_* = \Pi^{\mB}_{\cL}(x_0)$. Therefore, our methods solve the best approximation  problem
\begin{equation}\label{eq:s98y0f9yhfsee}\min_{x\in \R^n} \{\|x-x_0\|_{\mB}^2 \;:\; \mA x = b\}. \end{equation}
The ``dual'' of the basic method with {\em unit stepsize} was studied  in this context in \cite{SDA}. The Fenchel dual of the best approximation problem is an unconstrained concave quadratic maximization problem of the form $\max_{y\in \R^m} D(y),$
where the dual objective $D$ depends on $\mA, b, \mB$ and $x_0$. In \cite{SDA} it was shown that the basic method with unit stepsize  closely related to a {\em dual method} (stochastic dual subspace ascent) performing iterations of the form \begin{equation}\label{eq:SDA}y_{k+1} = y_k + \mS_k \lambda_k,\end{equation} where $\mS_k\sim \cD$, and $\lambda_k$ is chosen in such a way that the dual objective is as large as possible. Notice that the dual method in each iteration performs exact ``line search'' in a random subspace of $\R^m$ spanned by the columns of the random matrix $\mS_k$, and passing through $y_k$.  In particular, the iterates of the basic method with unit stepsize arise as affine images of the iterates of the dual method: $x_k = x_0 + \mB^{-1}\mA^\top y_k$. 

In a similar fashion, it is possible to interpret the methods developed in this paper as images of appropriately designed dual methods. In \cite{SDA}, the authors focus on establishing convergence of various quantities, such as dual objective, primal-dual gap, primal objective and so on. They obtain the complexity $1/\lambda_{\min}^+$, which is identical to the rate  we obtain here for the basic method with unit stepsize. However, their results require a stronger assumption on $\cD$ (their assumption implies exactness, but not vice versa). We perform a much deeper analysis from the novel viewpoint of stochastic reformulations of linear systems, include a stepsize, and propose and analyze parallel and accelerated variants.

In the special case when $\mS$ is chosen to be a random unit coordinate vector, \eqref{eq:SDA} specializes to  the {\em randomized coordinate descent method}, first analyzed by Leventhal and Lewis \cite{Leventhal:2008:RMLC}. 
In the special case when $\mS$ is chosen as a random column submatrix of the $m\times m$ identity matrix, \eqref{eq:SDA} specializes to the {\em randomized Newton method} of Qu et al.\ \cite{SDNA}. Randomized coordinate descent methods are the state of the art methods for certain classes of convex optimization problems with a very large number of variables. An analysis of an {\em asynchronous} randomized coordinate descent method for solving linear systems was performed in~\cite{ALS2015}. The first complexity analysis beyond quadratics was performed in \cite{ShalevTewari09, Nesterov:2010RCDM, UCDC}, a parallel method was developed in \cite{PCDM}, duality was explored in \cite{SDCA} and acceleration in \cite{APPROX}. 


%

\subsection{Randomized gossip algorithms} It was shown in \cite{SIMAX2015, Loizou-gossip-GlobalSIP} that for a suitable matrix $\mA$ encoding the structure of a graph, and for $b=0$, the application of the randomized Kaczmarz and randomized block Kaczmarz methods to \eqref{eq:s98y0f9yhfsee} lead to classical and new {\em randomized gossip algorithms} developed and studied in the signal processing literature, with new insights and novel proofs. Likewise, when applied in the same context, our new methods  lead to new parallel and accelerated gossip algorithms.

\subsection{Empirical risk minimization}

Regularized empirical risk minimization (ERM) problems are optimization problems of the form
\begin{equation}\label{eq:ERMxx}\min_{x\in \R^n} \frac{1}{m}\sum_{i=1}^m f_i(x) + g(x),\end{equation}
where $f_i$ is a loss function and $g$ a regularizer. Problems of this form are of key importance in machine learning and statistics \cite{Shai-book}. Let $f_i(x)=0$ if $\mA_{i:}x = b_i$ and $f_i(x)=+\infty$ otherwise, further  let $g(x) = \|x-x_0\|_{\mB}^2$. In this setting,  the ERM problem \eqref{eq:ERMxx} becomes equivalent to \eqref{eq:s98y0f9yhfsee}. While quadratic regularizers similar to $g$ are common in machine learning, zero/infinity loss functions are not used. For this reason, this specific instance of ERM was not studied in the machine learning literature. Since all our methods solve \eqref{eq:s98y0f9yhfsee}, they can be seen as stochastic algorithms for solving the ERM problem \eqref{eq:ERMxx}.

Since there is no reason to expect that any of our methods will satisfy $\mA x_k=b$ for any finite $k$,  the ERM objective value can remain to be equal to $+\infty$ throughout the entire iterative process. From this perspective, the value of the ERM objective is unsuitable as a measure of progress.

\subsection{Matrix inversion and quasi-Newton methods}

Given an invertible matrix $\mA$, its inverse  is the unique solution of the matrix equation $\mA \mX = \mI$. In \cite{inverse} the authors have extended the ``sketch and project'' method  \cite{SIMAX2015} to this equation. In each iteration of the method, one projects the current iterate matrix $\mX_k$, with respect to a weighted Frobenius norm, onto the sketched equation $\mS_k^\top \mA \mX = \mS_k^\top \mI$. This is a similar iterative process to the basic method with unit stepsize. The authors of \cite{inverse} prove that the iterates of method converge to  the inverse matrix at a linear rate, and detail connections of their method to quasi-Newton updates and approximate inverse preconditioning. A limited memory variant of the stochastic block BFGS method has been used to develop new efficient stochastic quasi-Newton methods for empirical risk minimization problems appearing in machine learning  \cite{SBFGS}.

It is possible to approach the problem $\mA \mX = \mI$ in the same way we approach the system \eqref{eq:linear_system} in our paper, writing down stochastic reformulations, and then developing new variants of the sketch and project method \cite{inverse}: a basic method with a stepsize, and parallel and accelerated methods. This would lead to the development of new variants of stochastic quasi-Newton rules, notably parallel and accelerated block BFGS. We conjecture that these rules will have superior performance to classical BFGS in practice.

Similar extensions and improvements can be done in relation to the problem of computing the pseudoinverse of very large rectangular matrices \cite{pseudoinverse}.

\section{Stochastic Reformulations of Linear Systems} \label{sec:reformulations}

In this section we formally derive the four stochastic formulations outlined in the introduction: {\em stochastic optimization}, {\em stochastic linear system}, {\em stochastic fixed point problem} and {\em probabilistic intersection}. Along the way we collect a number of results and observations which will be useful in the complexity analysis of our methods. 

\subsection{Projections} 

For a closed convex set $\emptyset\neq \cY\subseteq \R^n$, $\Pi^{\mB}_{\cY}$ denotes the projection operator onto $\cY$, in the $\mB$-norm:
$\Pi^{\mB}_{\cY}(x) \eqdef \arg \min_{y\in \R^n} \left\{\|y-x\|_{\mB} \;:\; y \in \cY\right\}.$ The $\mB$-pseudoinverse of a matrix $\mM$ is defined as
\begin{equation}\label{eq:B-pseudo}\mM^{\dagger_\mB} \eqdef \mB^{-1} \mM^\top (\mM \mB^{-1}\mM^\top)^\dagger.\end{equation}

The projection onto $\cL = \{x\;:\; \mA x = b\}$ is given by
\begin{equation}\label{eq:project_linear}\Pi^{\mB}_{\cL}(x) = x - \mB^{-1} \mA^\top  (\mA \mB^{-1} \mA^\top )^\dagger (\mA x-b) \overset{\eqref{eq:B-pseudo}}{=} x - \mA^{\dagger_\mB}(\mA x - b).\end{equation}

Note that for $\mB=\mI$, we get
$\mA^{\dagger_\mI} = \mA^\top (\mA \mA^\top)^\dagger = \mA^\dagger$, and hence the $\mI$-pseudoinverse reduces to the standard Moore-Penrose pseudoinverse. The $\mB$-pseudoinverse satisfies $\mA^{\dagger_\mB} b  =   \Pi^{\mB}_{\cL}(0) = \arg \min_{x} \{\|x\|_{\mB}\;:\; \mA x = b\}.$

\subsection{Stochastic functions}

Let $\cD$ be an arbitrary distribution over $m\times \sdim$ matrices. We shall write $\mS\sim \cD$ to say that $\mS$ is drawn from $\cD$. We shall often refer to   matrix expressions involving $\mS, \mA$ and $\mB$. In order to keep the expressions brief throughout the paper, it will be useful to define
\begin{equation}\label{eq:H}\mH \eqdef \mS (\mS^\top \mA \mB^{-1} \mA^\top \mS)^\dagger \mS^\top,\end{equation} and
\begin{equation}\label{eq:Z}\mZ \eqdef \mA^\top \mH \mA \overset{\eqref{eq:H}}{=} \mA^\top \mS (\mS^\top \mA \mB^{-1} \mA^\top \mS)^\dagger \mS^\top \mA.\end{equation} 

Notice that 
${\bf B}^{-1} {\bf Z} x = \arg \min_{y \in {\rm Range}({\bf B}^{-1} {\bf A}^\top {\bf S})} \|x- y\|_{\bf B}.$
That is, $\mB^{-1}\mZ$ is the projection matrix corresponding to  projection onto $\range{\mB^{-1} \mA^\top \mS}$ in the $\mB$-norm.
In particular, we have the relations
\begin{equation}\label{eq:ZBZ}
(\mB^{-1}\mZ)^2 = \mB^{-1}\mZ \qquad \text{and} \qquad \mZ \mB^{-1} \mZ = \mZ.
\end{equation}

Given $\mS\sim \cD$, we define the {\em stochastic (random) function}
\begin{equation} \label{eq:prodstoch}
f_{\mS} (x) \eqdef  \frac{1}{2}\norm{\mA x - b}_{\mH}^2 = \frac{1}{2}(\mA x - b)^\top \mH (\mA x - b).
\end{equation}

By combining \eqref{eq:prodstoch} and \eqref{eq:Z}, this can be also written in the form
\begin{equation} \label{eq:prodstoch2}
f_{\mS} (x)=  \frac{1}{2}(x-x_* )^\top \mZ(x - x_* ), \qquad x\in \R^n, \; x_*\in \cL.
\end{equation}

For all $\mS$ and all $x,h \in \R^n$ we have the expansion
$f_{\mS} (x + h) = f_{\mS} (x) + \dotprod{\nabla f_{\mS} (x), h}_{\mB}
+\frac{1}{2} \dotprod{(\nabla^2 f_{\mS})h, h}_{\mB}$,
where
\begin{equation} \label{eq:fder} \nabla f_{\mS} (x) =\mB^{-1}\mA^\top \mH (\mA x-b)\qquad
\text{and}  \qquad \nabla^2 f_{\mS}  = \mB^{-1} \mZ
\end{equation}
are the gradient and Hessian of $f_{\mS}$ with respect to the $\mB$-inner product, respectively.\footnote{If $\mB = \mI $, then $\dotprod{\cdot, \cdot}_{\mB}$ is the standard Euclidean inner product, and we recover formulas for the standard
gradient and Hessian. Note that $\mB^{-1} \mZ$ is both self-adjoint and positive semidefinite with respect to the $\mB$-inner product. Indeed, for all $x, y \in \R^n$ we have 
$\dotprod{\mB^{-1} \mZ x, y}_{\mB} = \dotprod{\mZ x, y}_{\mI} = \dotprod{x, \mZ y}_{\mI} = \dotprod{x, \mB^{-1} \mZ y}_{\mB}$, and
$\dotprod{\mB^{-1} \mZ x, x}_{\mB} = \dotprod{\mZ x, x}_{ \mI } \geq 0$.}   
In view of \eqref{eq:Z} and \eqref{eq:fder}, the gradient can also be written as
\begin{equation} \label{eq:fder2} \nabla f_{\mS} (x)  = \mB^{-1} \mZ(x-x_*), \qquad x\in \R^n, \; x_*\in \cL.
\end{equation}

Identities \eqref{eq:all_equal} in the following lemma explain why algorithm \eqref{eq:alg1-informal} can be equivalently written as stochastic gradient descent \eqref{eq:alg:SGD}, stochastic Newton method  \eqref{eq:alg:SNM}, stochastic fixed point method \eqref{eq:alg:SFP}, and stochastic projection method
\eqref{eq:alg:SPM}. For instance, the identity $(\nabla^2 f_{\mS}) \nabla f_{\mS}(x) = \nabla f_{\mS}(x)$ means that the stochastic gradients of $f_{\mS}$ are eigenvectors of the stochastic Hessian $\nabla^2 f_{\mS}$, corresponding to eigenvalue one. 

\begin{lemma} \label{lem:all_sort_of_stuff_is_equal}    For all $x\in \R^n$, we have
\begin{equation}\label{eq:all_equal}  \nabla f_{\mS}(x) = (\nabla^2 f_{\mS}) \nabla f_{\mS}(x)=(\nabla^2 f_{\mS} )^{\dagger_{\mB}}\nabla f_{\mS}(x)=x-\Pi^{\mB}_{\cL_{\mS}}(x)= \mB^{-1} \mA^\top  \mH (\mA x-b).\end{equation}
Moreover,
\begin{equation}\label{eq:v(x)}f_{\mS}(x) = \frac{1}{2}\|\nabla f_{\mS}(x)\|_{\mB}^2.\end{equation}
If $\cL_{\mS}$ is the set of minimizers of $f_{\mS}$, then $\cL\subseteq \cL_{\mS}$, and
\begin{enumerate}
\item[(i)]  $\cL_{\mS} = \{x\;:\; f_{\mS}(x)=0\} = \{x\;:\; \nabla f_{\mS}(x) = 0\}$
\item[(ii)] $\cL_{\mS} = x_* + \kernel{\mB^{-1}\mZ}$ for all $x_*\in \cL$
\item[(iii)] $\cL_{\mS} = \{x\;:\; \mB^{-1}\mA^\top \mH \mA x = \mB^{-1}\mA^\top \mH b\}$ \hfill see \eqref{eq:problem:preconditioning}
\item[(iv)] $\cL_{\mS} = \{x\;:\; \mS^\top \mA x = \mS^\top b \}$ \hfill see \eqref{eq:problem:intersection}
\end{enumerate}
Finally, for all $x\in \R^n$ we have the identity
\begin{equation}\label{eq:zero} f_{\mS}(x-\nabla f_{\mS}(x)) = 0.
\end{equation}

\end{lemma}

\begin{proof} Pick any $x_*\in \cL$. First, we have $\Pi^{\mB}_{\cL_{\mS}}(x) \overset{\eqref{eq:project_linear}}{=} x - \mB^{-1} \mA^\top  \mH (\mA x-b) \overset{\eqref{eq:fder}}{=} x-\nabla f_{\mS}(x)$. To establish \eqref{eq:all_equal}, it now only remains to consider the two expressions  involving the Hessian. We have
\[\nabla^2 f_{\mS} \nabla f_{\mS}(x) \overset{\eqref{eq:fder}+\eqref{eq:fder2}}{=} \mB^{-1}\mZ \mB^{-1}\mZ  (x-x_*) \overset{\eqref{eq:ZBZ}}{=} \mB^{-1}\mZ (x-x_*) \overset{\eqref{eq:fder2}}{=} \nabla f_{\mS}(x),\] and
\begin{eqnarray*}
 (\nabla^2 f_{\mS} )^{\dagger_{\mB}}\nabla f_{\mS}(x)
&\overset{\eqref{eq:B-pseudo}}{=}&  \mB^{-1} (\nabla^2 f_{\mS})^\top \left((\nabla^2 f_{\mS}) \mB^{-1} (\nabla^2 f_{\mS})^\top\right)^\dagger \nabla f_{\mS}(x)\\
&\overset{\eqref{eq:fder}}{=} & \mB^{-1}  (\mB^{-1}\mZ )^\top \left((\mB^{-1}\mZ ) \mB^{-1} (\mB^{-1}\mZ )^\top\right)^\dagger \mB^{-1}\mZ  (x-x_*)\\
 &=& \mB^{-1} \mZ  \mB^{-1} \left(\mB^{-1}\mZ  \mB^{-1} \mZ  \mB^{-1}\right)^\dagger \mB^{-1}\mZ  (x-x_*)\\
 &\overset{\eqref{eq:ZBZ}}{=}& \left(\mB^{-1} \mZ \mB^{-1}\right) \left(\mB^{-1}\mZ   \mB^{-1}\right)^\dagger \left(\mB^{-1}\mZ  \mB^{-1} \right) \mB(x-x_*)\\
 &=&  \mB^{-1} \mZ   (x-x_*) \quad \overset{\eqref{eq:fder2}}{=} \quad \nabla f_{\mS}(x).
\end{eqnarray*}
Identity  \eqref{eq:v(x)} follows from
\[ \frac{1}{2}\|\nabla f_{\mS}(x)\|_{\mB}^2 \overset{\eqref{eq:fder2}}{=} \frac{1}{2}(x-x_*)^\top  \mZ \mB^{-1} \mZ(x-x_*)\overset{\eqref{eq:ZBZ}}{=} \frac{1}{2}(x-x_*)^\top  \mZ(x-x_*)\overset{\eqref{eq:prodstoch2}}{=} f_{\mS}(x).\]

If $x\in \cL$, then by picking $x_*=x$ in \eqref{eq:fder2}, we see that $x\in \cL_{\mS}$.
It remains to show that the sets defined in (i)--(iv) are identical. Equivalence between (i) and (ii) follows from \eqref{eq:fder2}. Now consider (ii) and (iii). Any $x_*\in \cL$ belongs to the set defined in (iii), which follows immediately by substituting  $b = \mA x_*$. The rest follows after observing the nullspaces are identical. In order to show that (iii) and (iv) are equivalent, it suffices to compute $\Pi^{\mB}_{\cL_{\mS}}(x)$ and observe that $\Pi^{\mB}_{\cL_{\mS}}(x) = x$ if and only if $x$ belongs to the set defined in (iii).

It remains to  establish \eqref{eq:zero}. In view of (i), it suffices to show that $x - \nabla f_{\mS}(x)\in \cL_{\mS}$. However, from \eqref{eq:all_equal} we know that $x - \nabla f_{\mS}(x)= \Pi^\mB_{\cL_{\mS}}(x)\in \cL_{\mS}$.
\end{proof}


\subsection{Stochastic reformulation}

In order to proceed, we shall enforce a basic assumption on $\cD$.

\begin{assumption}[Finite mean]\label{ass:H_finite}
The random matrix $\mH$ has a mean. That is, the matrix
$\EE{\mS\sim \cD}{\mH } = \EE{\mS\sim \cD}{\mS(\mS^\top \mA \mB^{-1}\mA^\top \mS)^\dagger \mS^\top}
$ has finite entries.
\end{assumption}

This is an assumption on $\cD$ since a suitable distribution satisfying it exists for all $\mA \in \R^{m\times n}$ and $\mB\succ 0$. Note that if the assumption holds, then $\Exp{\mH}$ is symmetric and positive semidefinite.  We shall enforce this assumption throughout the paper and hence will not henceforth refer to it.

\begin{example} Let $\cD$ be the uniform distribution over unit basis vectors in $\R^m$. That is, $\mS = e_i$ (the $i$th unit basis vector in $\R^m$) with probability $1/m$. Then 
\[\Exp{\mH} = \sum_{i=1}^m \frac{1}{m} e_i (\mA_{i:} \mB^{-1} \mA_{i:}^\top)^\dagger e_i^\top = \frac{1}{m}\Diag{\alpha_1,\dots,\alpha_m },\]
where $\alpha_i=1/\|\mA_{i:}^\top\|_{\mB^{-1}}^{2}$ for $i=1,2,\dots,m$. If  $\mA$ has nonzero rows, then $\Exp{\mH}\succ 0$.
\end{example}

In this paper we reformulate the linear system~\eqref{eq:linear_system} as the {\em stochastic  optimization problem}
\begin{equation} \label{eq:min_f}
\boxed{\min_{x\in \R^n} \left\{f (x) \eqdef \EE{\mS \sim \cD}{f_\mS (x)} \right\}}
\end{equation}

Under Assumption~\ref{ass:H_finite}, it is straightforward to check that the expectation in \eqref{eq:min_f} is finite for all $x$, and hence $f$ is well defined. The following is a direct consequence of Lemma~\ref{lem:all_sort_of_stuff_is_equal}. We shall use these formulas throughout the paper.

\begin{lemma}[Representations of $f$] \label{lem:f-various} Function $f$ defined in \eqref{eq:min_f} can be represented in multiple ways:
$f(x) =  \frac{1}{2}\Exp{\|x-\Pi^{\mB}_{\cL_{\mS}}(x)\|_{\mB}^2}= \frac{1}{2}\Exp{\|\nabla f_{\mS}(x)\|_{\mB}^2}.$
Moreover, 
\begin{equation}\label{eq:probminf}
f (x)  \overset{\eqref{eq:prodstoch}}{=} \frac{1}{2}\norm{\mA x -b}_{\Exp{\mH}}^2 = \frac{1}{2} (\mA x - b)^\top\Exp{\mH} (\mA x- b),
\end{equation}
and for any $x_*\in \cL$ we can write
\begin{equation}\label{eq:probminf2}
f (x) =
\frac{1}{2}(x -x_* )^\top \Exp{\mZ}(x - x_*).
\end{equation}
\end{lemma}

Since $\Exp{\mH}\succeq 0$, $f$ is a convex quadratic function. Moreover, $f$ is nonnegative. 

\begin{lemma}\label{lem:interchange} 
We have the identities $\nabla \Exp{f_{\mS}(x)} = \Exp{\nabla f_{\mS}(x)}$ and $\nabla^2 \Exp{f_{\mS}(x)} = \Exp{\nabla^2 f_{\mS}(x)}= \mB^{-1}\Exp{\mZ}$ for all $x\in \R^n$. That is, differentiation and expectation can be interchanged.
\end{lemma}
\begin{proof}
First note that ${\rm E}\left[f_{\mS}(x)\right] \overset{\eqref{eq:problem:stoch_opt}}{=}  f(x) \overset{\eqref{eq:probminf}}{=} \frac{1}{2} (\mA x - b)^\top {\rm E}[\mH](\mA x - b )$. Therefore, the gradient of $f$ with respect to the $\mB$--inner product is $\nabla f(x) = \mB^{-1} \mA^\top {\rm E}[\mH] (\mA x - b).$ Using this, by linearity of expectation we get
\[ {\rm E} \left[\nabla f_{\mS} (x) \right]  \overset{\eqref{eq:fder}}{=} \mB^{-1} \mA^\top {\rm E} \left[ \mH\right] (\mA x - b) = \nabla f(x) \overset{\eqref{eq:problem:stoch_opt}}{=}  \nabla  {\rm E}\left[ f_{\mS}(x) \right].\] The claim about the Hessians follows the same reasoning.
\end{proof}

In view of the above lemma, the gradient and Hessian of $f$ (with respect to the $\mB$-inner product) are given by
\begin{equation}\label{eq:grad_f}
\nabla f (x) = \EE{\mS \sim \cD}{ \nabla f_\mS (x)} \overset{\eqref{eq:fder2} }{=} \mB^{-1} \Exp{\mZ}(x- x_* ), \; \nabla^2 f =\Exp{\nabla^2 f_{\mS}} = \mB^{-1}\Exp{\mZ}, 
\end{equation}
respectively, where $x_*$ is any point in $\cL$.

The set of minimizers of $f$, denoted $\cX$, can be  represented in several ways, as captured by our next result.  It immediately follows  that the four stochastic formulations mentioned in the introduction are equivalent. 

\begin{theorem}[Equivalence of stochastic formulations] \label{thm:X} Let $x_*\in \cL$. The following sets are identical:  \begin{enumerate}
\item[(i)] $\cX = \arg \min f(x) = \{x\;:\; f(x) = 0\}= \{x\;:\; \nabla f(x) = 0\}$  \hfill see \eqref{eq:problem:stoch_opt}
\item[(ii)] $\cX = \{x\;:\; \mB^{-1} \mA^\top \Exp{\mH} \mA x = \mB^{-1} \mA^\top \Exp{\mH} b\} = x_*+ \kernel{\Exp{\mZ}}$  \hfill see \eqref{eq:problem:preconditioning}
\item[(iii)] $\cX  = \{x\;:\; \Exp{\Pi^{\mB}_{\cL_{\mS}}(x)} = x\}$ \hfill see \eqref{eq:problem:fixed_point}
\item[(iv)] $\cX= \{x\;:\; \Prob(x\in \cL_{\mS}) = 1\}$ \hfill see \eqref{eq:problem:intersection}
\end{enumerate}
As a consequence, the stochastic problems
\eqref{eq:problem:stoch_opt}, \eqref{eq:problem:preconditioning},
\eqref{eq:problem:fixed_point}, and
\eqref{eq:problem:intersection}
 are equivalent (i.e., their solutions sets are identical). Moreover, the set $\cX$ does not depend on $\mB$.
\end{theorem}

\begin{proof} 
As $f$ is convex, nonnegative and achieving the value of zero (since $\cL\neq \emptyset$), the sets in (i) are all identical. We shall now show that the sets defined in (ii)--(iv) are equal to that defined in (i). Using the formula for the gradient from  \eqref{eq:grad_f}, we  see that $\{x\;:\; \nabla f(x) = 0\}= \{x\;:\; \mB^{-1}\Exp{\mZ}(x-x_*) = 0\} = \{x\;:\; \Exp{\mZ}(x-x_*) = 0\} = x_* + \{h \;:\; \Exp{\mZ}h = 0\} = x_* + \kernel{\Exp{\mZ}}$, which shows that (i) and (ii) are the same. Equivalence of (i) and (iii) follows by taking expectations in $\eqref{eq:all_equal}$ to obtain $\nabla f(x)  = \Exp{\nabla f_{\mS}(x)} \overset{\eqref{eq:all_equal}}{=} \Exp{x-\Pi^{\mB}_{\cL_{\mS}}(x)}$.

It remains to establish equivalence between (i) and (iv). Let  
\begin{equation}\label{eq:suhh0sys}\cX = \{x\;:\; f(x)=0\} \overset{\text{(Lemma~\ref{lem:f-various})}}{=} \left\{x\;:\; \Exp{\left\|x- \Pi^{\mB}_{\cL_{\mS}}(x)\right\|_{\mB}^2} = 0\right\}\end{equation}
and let $\cX'$ be the set from (iv).
For easier reference, let $\xi_{\mS}(x) \eqdef \left\|x- \Pi^{\mB}_{\cL_{\mS}}(x)\right\|_{\mB}^2$.  The following three probabilistic events are identical: 
\begin{equation}\label{eq:equiv-events}\left[ x\in \cL_{\mS} \right] =  \left[ x=\Pi^{\mB}_{\cL_{\mS}}(x) \right] =  \left[\xi_{\mS}(x) = 0\right].\end{equation} Therefore, if $x\in \cX'$, then the random variable $\xi_{\mS}(x)$ is equal to zero with probability 1, and hence $x\in \cX$. Let us now establish the reverse inclusion. First, let $1_{[\xi_{\mS}(x) \geq t]}$ be the indicator function of the event $[\xi_{\mS}(x) \geq t]$. Note that since $\xi_{\mS}(x)$ is a nonnegative random variable, for all $t \in \R$ we have the inequality \begin{equation}\label{eq:kjshbnvb}\xi_{\mS}(x) \geq t 1_{\xi_{\mS}(x) \geq t}.\end{equation}
Now take $x\in \cX$ and consider  $t>0$. By taking expectations in \eqref{eq:kjshbnvb}, we obtain
\[0=\Exp{\xi_{\mS}(x)} \geq \Exp{t 1_{\xi_{\mS}(x) \geq t}} = t \Exp{1_{\xi_{\mS}(x) \geq t}} = t \Prob(\xi_{\mS}(x) \geq t),\]
which implies that $\Prob(\xi_{\mS}(x)\geq t) = 0$. Now choose $t_i = 1/i$ for $i=1,2,\dots$ and note that the event $[\xi_{\mS}(x) > 0]$ can be written as
$[\xi_{\mS}(x) > 0] = \bigcup_{i=1}^{\infty} \; [\xi_{\mS}(x) \geq t_i].$
Therefore, by the union bound, $\Prob(\xi_{\mS}(x) > 0) = 0$, which immediately implies that $\Prob(\xi_{\mS}(x) = 0) = 1$. From \eqref{eq:equiv-events} we conclude that $x\in \cX'$.

That $\cX$ does not depend on $\mB$ follows from representation (iv).
\end{proof}
 

\subsection{Exactness of the Reformulations} \label{sec:exactness}

In this section ask the following question: when are the stochastic formulations \eqref{eq:problem:stoch_opt}, \eqref{eq:problem:preconditioning}, \eqref{eq:problem:fixed_point},  \eqref{eq:problem:intersection} equivalent to the  linear system \eqref{eq:linear_system}? This leads to the concept of {\em exactness}, captured by the following assumption. 

\begin{assumption}[Exactness] \label{ass:main} Stochastic reformulations \eqref{eq:problem:stoch_opt}, \eqref{eq:problem:preconditioning}, \eqref{eq:problem:fixed_point},  \eqref{eq:problem:intersection}  of problem \eqref{eq:linear_system} are {\em exact}. 
That is,   $\cX=\cL$.   
\end{assumption}

We do not need  this assumption for all our results, and hence we will specifically invoke it when needed. 
For future reference in the paper, it will be useful to be able to draw upon several equivalent characterizations of exactness. 

\begin{theorem}[Exactness] \label{thm:exactness} The following statements are equivalent:
\begin{enumerate}
\item[(i)] Assumption~\ref{ass:main} holds

\item[(ii)] $\kernel{\Exp{\mZ}} = \kernel{\mA}$


\item[(iii)] $\kernel{\mB^{-1/2}\Exp{\mZ}\mB^{-1/2}} = \kernel{\mA \mB^{-1/2}}$

\item[(iv)] $\range{\mA} \cap \kernel{\Exp{\mH}} = \{0\}$ 


\end{enumerate}
\end{theorem}

\begin{proof} 
Choose any $x_*\in \cL$. We know that $\cL = x_* + \kernel{\mA}$. On the other hand, Theorem~\ref{thm:X} says that $\cX = x_* + \kernel{\Exp{\mZ}}$. This establishes equivalence of (i) and (ii). If (ii) holds, then $ \kernel{\mA}=\kernel{\Exp{\mZ}}=\kernel{\mB^{-1/2}\Exp{\mZ}}$, and (iii) follows. If (iii) holds, then $\kernel{\mA} = \kernel{\mB^{-1/2}\Exp{\mZ}} = \kernel{\Exp{\mZ}}$, proving (ii). We now show that (ii) and (iv) are equivalent. Since $\Exp{\mZ} = \mA^\top (\Exp{\mH})^{1/2} (\Exp{\mH})^{1/2}  \mA$, we have $\kernel{\Exp{\mZ}} = \kernel{(\Exp{\mH})^{1/2}  \mA}$. Moreover,  a) $\kernel{(\Exp{\mH})^{1/2}  \mA} = \kernel{\mA}$ if and only if $\range{\mA}\cap \kernel{(\Exp{\mH})^{1/2}} = \{0\}$, and b) $\kernel{(\Exp{\mH})^{1/2}} = \kernel{\Exp{\mH}}$. It remains to combine these observations.
\end{proof}

 We now list two sufficient conditions for exactness.
 
\begin{lemma}[Sufficient conditions]
 Any of these conditions implies that Assumption~\ref{ass:main}  is satisfied:
\begin{enumerate}
 \item[(i)] $\Exp{\mH } \succ 0$
 \item[(ii)] $\kernel{\Exp{\mH}} \subseteq \kernel{\mA^\top}$
\end{enumerate}
\end{lemma}
\begin{proof} If (i) holds, then $\kernel{\Exp{\mZ}} = \kernel{\mA^\top \Exp{\mH} \mA} = \kernel{\mA}$, we have exactness by applying Theorem~\ref{thm:exactness}.  Finally, (ii) implies statement (iv) in Theorem~\ref{thm:exactness}, and hence exactness follows.
\end{proof}

\section{Basic Method} \label{sec:Basic_Method}

We propose solving~\eqref{eq:probminf} by Algorithm~\ref{alg:alg1}.

\begin{algorithm}[!h]
\begin{algorithmic}[1]
\State \textbf{Parameters:} distribution $\cD$ from which to sample matrices; positive definite matrix $\mB\in \R^{n \times n}$; stepsize/relaxation parameter $\omega\in \R$
\State Choose $x_0 \in \R^n$
\Comment Initialization
\For {$k = 0, 1, 2, \dots$}
	\State Draw a fresh sample $\mS_k \sim \cD$
	\State Set $x_{k+1} = x_k - \omega \mB^{-1} \mA^\top  \mS_k (\mS_k^\top  \mA \mB^{-1} \mA^\top  \mS_k)^\dagger \mS_k^\top  (\mA x_k - b) $			
\EndFor
\end{algorithmic}
\caption{Basic Method}
\label{alg:alg1}
\end{algorithm}

\begin{remark}Since $\mS$ is random, the matrices $\mH=\mS(\mS^\top \mA \mB^{-1} \mA^\top \mS)^\dagger \mS^\top$ and $\mZ = \mA^\top \mH \mA$ are also random. At iteration $k$ of our algorithm, we  sample matrix $\mS_k\sim \cD$ and perform an update step. It will be useful to to use the notation $\mH_k=\mS_k(\mS_k^\top \mA \mB^{-1} \mA^\top \mS_k)^\dagger \mS_k^\top$ and $\mZ_k = \mA^\top \mH_k \mA$. In this notation, Algorithm~\ref{alg:alg1} can be written in the form $x_{k+1}=x_k - \omega \mB^{-1} \mA^\top \mH_k (\mA x_k -b)$.
\end{remark}

In the rest of this section we analyze Algorithm~\ref{alg:alg1}.

\subsection{Condition number of the stochastic reformulation}

Recall that the Hessian of $f$  is given by
\begin{equation}
\nabla^2 f  = \EE{\mS \sim \cD}{\nabla^2 f_\mS } = \mB^{-1} \Exp{\mZ}.
\end{equation}
Since $\mB^{-1} \Exp{\mZ}$ is not symmetric (although it is self-adjoint with respect to the $\mB$-inner product), it will be more convenient to instead study the spectral properties of the related matrix
$\mB^{-1/2}\Exp{\mZ}\mB^{-1/2}$. Note that this matrix is symmetric, and has the same spectrum as $\mB^{-1}\Exp{\mZ}$. Let 
\begin{equation}\label{eq:eig_decomp}\mW\eqdef \mB^{-1/2}\Exp{\mZ} \mB^{-1/2} = \mU \Lambda \mU^\top  = \sum_{i=1}^n \lambda_i u_i u_i^\top  
\end{equation} 
be the eigenvalue decomposition of $\mW$, where $\mU = [u_1,\dots,u_n]\in \R^{n\times n}$ is an orthonormal matrix composed of eigenvectors, and $\Lambda = \Diag{\lambda_1,\lambda_2,\dots,\lambda_n}$ is a diagonal matrix of eigenvalues.  Assume without loss of generality that the eigenvalues are ordered from largest to smallest: $ \lambda_1\geq \lambda_2 \geq \cdots \geq  \lambda_{n} \geq 0$. We shall often write $\lambda_{\max}=\lambda_1$ to denote the largest eigenvalue, and $\lambda_{\min} = \lambda_n$ for the smallest eigenvalue.


\begin{lemma}\label{lem:spectrumxx}  $0\leq \lambda_i\leq 1$ for all $i$.
\end{lemma}
\begin{proof} Since $\mB^{-1/2}\mZ \mB^{-1/2}$ is symmetric positive semidefinite, so is its expectation $\mW$, implying that $\lambda_i\geq 0$ for all $i$. Further, note that $\mB^{-1/2}\mZ \mB^{-1/2}$ is a projection matrix. Indeed, it is the projection (in the standard $\mI$-norm) onto $\range{\mB^{-1/2}\mA^\top \mS}$. Therefore, its eigenvalues are all zeros or ones. Since the map $\mX\mapsto \lambda_{\max}(\mX)$ is convex, by Jensen's inequality we get \[\lambda_{\max}(\mW) = \lambda_{\max}\left(\Exp{\mB^{-1/2}\mZ \mB^{-1/2}}\right) \leq \Exp{\lambda_{\max}(\mB^{-1/2}\mZ \mB^{-1/2})} \leq 1.\]
\end{proof}

It follows from Assumption~\ref{ass:main}  that $\lambda_{\max}>0$. Indeed, if we assume that  $\lambda_i = 0$ for all $i$, then from Theorem~\ref{thm:exactness} and the fact that $\kernel{\mW} = \range{u_i \;:\; \lambda_i=0}$ we conclude that $\kernel{\mA \mB^{-1/2}} = \R^n$, which in turn implies  that $\kernel{\mA} = \R^n$.  This can only happen if $\mA = 0$, which is a trivial case we excluded from consideration in this paper by assumption.

Now, let $j$ be the largest index for which $\lambda_j>0$.  We shall often write $\lambda_{\min}^+=\lambda_j$. If all eigenvalues $\{\lambda_i\}$ are positive, then $j=n$.

We now  define the {\em condition number} of problem \eqref{eq:min_f} to be the quantity
\begin{equation}\label{eq:kappa}\condnum \eqdef \|\mW\| \|\mW^\dagger\| = \frac{\lambda_{\max}}{\lambda_{\min}^+}.\end{equation}

\begin{lemma}[Quadratic bounds] \label{lem:sandwich} For all $x\in \R^n$ and $x_*\in \cL$ we have 
\begin{equation} \label{eq:sandwich}\lambda_{\min}^+ \cdot  f(x) \leq \frac{1}{2}\|\nabla f(x)\|_{\mB}^2 \leq \lambda_{\max} \cdot f(x),\end{equation}
\begin{equation} \label{eq:jshvs5r6989ss}f(x) \leq \frac{\lambda_{\max}}{2} \|x -x_*\|_{\mB}^2.
\end{equation}
Moreover, if Assumption~\ref{ass:main} holds, then  for all  $x\in \R^n$ and  $x_*=\Pi^{\mB}_{\cL}(x)$ we have
\begin{equation}\label{eq:jsguygusopoipoip}  \frac{\lambda_{\min}^+}{2} \|x -x_*\|_{\mB}^2 \leq f(x).\end{equation}
\end{lemma}

\begin{proof} 
In view of \eqref{eq:probminf2} and \eqref{eq:eig_decomp}, we obtain a spectral characterization of $f$:
\begin{equation}\label{eq:f_spectral_decomp}f (x) =
\frac{1}{2} \sum_{i=1}^n \lambda_i \left(u_i^\top \mB^{1/2}(x -x_* )\right)^2,\end{equation}
where $x_*$ is any point in $\cL$. On the other hand,  in view of \eqref{eq:grad_f} and \eqref{eq:eig_decomp}, we have 
\begin{eqnarray}\label{eq:h98gs98g98s00033}\|\nabla f(x)\|_{\mB}^2 &=& \|\mB^{-1}\Exp{\mZ}(x-x_*)\|_{\mB}^2 = (x-x_*)^\top \Exp{\mZ} \mB^{-1} \Exp{\mZ} (x-x_*)\notag \\
&=&(x-x_*)^\top \mB^{1/2} ( \mB^{-1/2}\Exp{\mZ} \mB^{-1/2}) (\mB^{-1/2} \Exp{\mZ} \mB^{-1/2}) \mB^{1/2} (x-x_*)\notag \\
&=&(x-x_*)^\top \mB^{1/2} \mU (\mU^\top \mB^{-1/2}\Exp{\mZ} \mB^{-1/2} \mU) (\mU^\top \mB^{-1/2} \Exp{\mZ} \mB^{-1/2} \mU) \notag \\ 
&& \quad \mU^\top \mB^{1/2} (x-x_*)\notag \\
&\overset{\eqref{eq:eig_decomp}}{=}& (x-x_*)^\top \mB^{1/2} \mU \Lambda^2 \mU^\top \mB^{1/2} (x-x_*)\notag \\
&=& \sum_{i=1}^n \lambda_i^2  \left(u_i^\top \mB^{1/2}(x -x_* )\right)^2.
\end{eqnarray}
Inequality \eqref{eq:sandwich} follows by comparing \eqref{eq:f_spectral_decomp} and \eqref{eq:h98gs98g98s00033}, using the bounds $\lambda_{\min}^+ \lambda_i \leq  \lambda_i^2 \leq \lambda_{\max} \lambda_i$, which hold for $i$ for which $\lambda_i>0$. 

We now move to the bounds involving norms. First, note that for any $x_*\in \cL$, 
\begin{equation} \label{eq:hhgsfsrppls} f(x)\overset{\eqref{eq:probminf2}}{=} 
 \frac{1}{2}(\mB^{1/2}(x-x_*))^\top (\mB^{-1/2}\Exp{\mZ}\mB^{-1/2}) \mB^{1/2} (x-x_*).
\end{equation}
The upper bound follows by applying the inequality $\mB^{-1/2}\Exp{\mZ}\mB^{-1/2} \preceq \lambda_{\max}\mI$. If $x_* = \Pi^{\mB}_{\cL}(x)$, then
in view of \eqref{eq:project_linear}, we  have $\mB^{1/2}(x-x_*) \in \range{\mB^{-1/2}\mA^\top}$. Applying Lemma~\ref{lem:WGWtight2} to \eqref{eq:hhgsfsrppls}, we get the lower bound.
\end{proof}

\begin{remark}
Bounds such as those  in Lemma~\ref{lem:sandwich} are often seen in convex optimization. In particular, if $\phi:\R^n\to \R$ is a $\mu$-strongly convex and $L$-smooth function, then 
$\mu(\phi(x)-\phi^*) \leq \frac{1}{2}\|\nabla \phi(x)\|^2 \leq L (\phi(x)-\phi^*)$ for all $x\in \R^n$, where $\phi^* = \min_{x} \phi(x)$. In our case, the optimal objective value is zero. The presence of $\mB$-norm is due to us defining gradients using the $\mB$-inner product. Moreover, it is the case that $f$ is $\lambda_{\max}$-smooth, which explains the upper bound. However, $f$ is not necessarily $\mu$-strongly convex for any $\mu>0$, since $\Exp{\mZ}$ is not necessarily positive definite. However, we still obtain a nontrivial lower bound. 
\end{remark}

\subsection{Convergence of expected iterates}

We now present a  fundamental theorem  precisely describing the evolution of the expected iterates of the basic method. 

\begin{theorem}[Convergence of expected iterates] \label{thm:alg1_complexity_first} Choose any $x_0\in \R^n$ and let
$\{x_k\}$ be the random iterates produced by Algorithm~\ref{alg:alg1}.
\begin{enumerate}
\item Let  $x_*\in \cL$ be chosen arbitrarily.  Then 
\begin{equation}\label{eq:Eiterates}
\Exp{x_{k+1}- x_* } =\left( \mI - \omega\mB^{-1} \Exp{\mZ}\right) \Exp{x_k- x_* }.\end{equation}
Moreover, by transforming the error via the linear mapping $h \to \mU^\top \mB^{1/2}h$, this can be written in the form
\begin{equation}\label{eq:9898gifkdmd}\Exp{\mU^\top  \mB^{1/2}(x_k-x_*)} = (\mI -  \omega  \Lambda )^k \mU^\top  \mB^{1/2}(x_0-x_*),
\end{equation} 
which is separable in the coordinates of the transformed error:
\begin{equation}\label{eq:98g9dgisss}\Exp{u_i^\top  \mB^{1/2}(x_k-x_*)} = (1 -  \omega  \lambda_i )^k u_i^\top  \mB^{1/2}(x_0-x_*), \quad i=1,2,\dots,n.\end{equation}
Finally,
\begin{equation}\label{eq:h98g98ishihfp}\|\Exp{x_k-x_*}\|_{\mB}^2 = \sum_{i=1}^n (1-\omega \lambda_i)^{2k} \left(u_i^\top  \mB^{1/2}(x_0-x_*)\right)^2 .\end{equation}

\item Assumption~\ref{ass:main} hold and let $x_* = \Pi^{\mB}_{\cL}(x_0)$. Then  for all $i=1,2,\dots,n$,
\begin{equation}\label{eq:s98g98gjshis}\Exp{u_i^\top  \mB^{1/2}(x_k-x_*)} = \begin{cases}
0 &\text{if} \quad \lambda_i=0,\\
(1 -  \omega  \lambda_i )^k u_i^\top  \mB^{1/2}(x_0-x_*) & \text{if} \quad \lambda_i>0.
\end{cases} \end{equation}
Moreover,
\begin{equation}\label{eq:ug98fg98bdd}\|\Exp{x_k-x_*}\|_{\mB}^2 \leq \rho^k(\omega) \|x_0-x_*\|_{\mB}^2,\end{equation}
where the rate is given by \begin{equation}\label{eq:rho}
\rho(\omega) \eqdef \max_{i: \lambda_i>0} (1-\omega \lambda_i)^2.
\end{equation}
\end{enumerate}
\end{theorem}

Note that all eigenvalues of $\mW$ play a role, governing the convergence speeds of individual elements of the transformed error vector. Under the assumption of exactness,  and relative to the particular solution $x_*=\Pi^{\mB}_{\cL}(x_0)$, the expected error $\Exp{x_k-x_*}$ converges to zero at a linear rate.  The proof of the theorem is provided in Section~\ref{sec:proof_xxx}.

\begin{remark}Having established \eqref{eq:Eiterates}, perhaps the the most obvious way of analyzing the method is by taking $\mB$-norms on both sides of identity \eqref{eq:Eiterates}. This way we obtain the estimate  
$\|\Exp{x_{k+1}-x_*}\|_{\mB}^2\leq \tilde{\rho}(\omega) \|\Exp{x_{k}-x_*}\|_{\mB}^2,$
where
$\tilde{\rho}(\omega) = \norm{\mI - \omega\mB^{-1} \Exp{\mZ}}_{\mB}^2
= \sigma^2_{\max} (\mI - \omega\mB^{-1/2} \Exp{\mZ}\mB^{-1/2} ),$
$ \norm{\mM}_{\mB}
\eqdef \max \{\norm{{\mM} x}_{\mB} \;:\; \norm{x}_{\mB} \leq 1\}$, and $\sigma_{\max} (\cdot)$ denotes the largest singular value. This gives the inequality
\begin{equation}\label{eq:alternative_xx}\|\Exp{x_{k}-x_*}\|_{\mB}^2 \leq \tilde{\rho}^{k}(\omega) \|x_{0}-x_*\|_{\mB}^2,\end{equation}
which can be directly compared with \eqref{eq:h98g98ishihfp}. We now highlight two differences between these two bounds. The first approach gives a more  detailed, information, as the identity in \eqref{eq:h98g98ishihfp} is an exact formula for the norm of the expected error. Moreover, while in view of  \eqref{eq:rho}, we have $\rho(\omega) =  \max_{i: \lambda_i>0} (1-\omega \lambda_i)^2$, it can be shown that $\tilde{\rho}(\omega) = \max_{i} (1-\omega \lambda_i)^2$. The two bounds are identical if $\lambda_{\min}>0$, but they differ otherwise. In particular, as long as $\lambda_{\min}=0$, we have $\tilde{\rho}(\omega)\geq 1$ for all $\omega$, which means that the bound \eqref{eq:alternative_xx} does not guarantee convergence.

\end{remark}

The following result, characterizing convergence of the expected errors to zero, is a straightforward corollary of Theorem~\ref{thm:alg1_complexity_first}.

\begin{corollary}[Necessary and sufficient conditions for convergence]\label{thm:corollary}
Let Assumption~\ref{ass:main} hold. Choose any $x_0\in \R^n$ and let $x_*=\Pi^{\mB}_{\cL}(x_0)$. If $\{x_k\}$ are the random iterates produced by Algorithm~\ref{alg:alg1}, then the following statements are equivalent:
\begin{enumerate}
\item[(i)] $|1-\omega \lambda_i|<1$ for all $i$ for which $\lambda_i>0$
\item[(ii)] $0< \omega< 2/ \lambda_{\max}$
\item[(iii)] 
$\Exp{u_i^\top  \mB^{1/2}(x_k-x_*)}\to 0$ for all $i$
\item[(iv)] $\|\Exp{x_k-x_*}\|_{\mB}^2 \to 0$
\end{enumerate}
\end{corollary}

\subsection{Proof of Theorem~\ref{thm:alg1_complexity_first}}
\label{sec:proof_xxx}

We first start with a lemma.

\begin{lemma} \label{lem:whath_appens_with_zero_eigenvalues}Let Assumption~\ref{ass:main} hold. Consider arbitrary $x\in \R^n$ and let  $x_* = \Pi^{\mB}_{\cL}(x)$. If  $\lambda_i=0$, then $u_i^\top  \mB^{1/2}(x-x_*)=0$.
\end{lemma}
\begin{proof} From \eqref{eq:project_linear} we see that $x-x_* = \mB^{-1} \mA^\top  w$ for some $w\in \R^m$. Therefore, $u_i^\top  \mB^{1/2}(x-x_*) = u_i^\top  \mB^{-1/2} \mA^\top  w$. By Theorem~\ref{thm:exactness}, we have
$\range{u_i \;:\; \lambda_i=0} = \kernel{\mA \mB^{-1/2}}$,
from which it follows that $u_i^\top  \mB^{-1/2} \mA^\top = 0$.
\end{proof}

We now proceed with the proof of Theorem~\ref{thm:alg1_complexity_first}. The iteration of Algorithm \ref{alg:alg1} can be written in the form \begin{equation}\label{eq:9898gsgusss}e_{k+1} = (\mI-\omega \mB^{-1}\mZ_k)e_k,\end{equation}
where $e_k=x_{k}-x_*$.   Multiplying both sides of this equation by $\mB^{1/2}$  from the left, and taking expectation conditional on $e_k$, we obtain
\[\Exp{\mB^{1/2}e_{k+1} \;|\; e_k} = (\mI - \omega \mB^{-1/2}\Exp{\mZ}\mB^{-1/2}) \mB^{1/2}e_k.\]
Taking expectations on both sides and using the tower property, we get 
\[\Exp{\mB^{1/2}e_{k+1} } = \Exp{\Exp{\mB^{1/2}e_{k+1} \;|\; e_k}} = (\mI - \omega \mB^{-1/2}\Exp{\mZ}\mB^{-1/2}) \Exp{\mB^{1/2}e_k}.\]
We now replace $\mB^{-1/2}\Exp{\mZ}\mB^{-1/2}$ by its eigenvalue decomposition $\mU\Lambda \mU^\top $ (see \eqref{eq:eig_decomp}), multiply both sides of the last inequality by $\mU^\top $ from the left, and use linearity of expectation to obtain
\[\Exp{\mU^\top  \mB^{1/2}e_{k+1}} = (\mI - \omega  \Lambda ) \Exp{\mU^\top  \mB^{1/2}e_k}.\]
Unrolling the recurrence, we get \eqref{eq:9898gifkdmd}. When this is written coordinate-by-coordinate, \eqref{eq:98g9dgisss} follows. Identity \eqref{eq:h98g98ishihfp} follows immediately by equating standard Euclidean norms of both sides of  \eqref{eq:9898gifkdmd}. If $x_*=\Pi^{\mB}_{\cL}(x_0)$, then from Lemma~\ref{lem:whath_appens_with_zero_eigenvalues} we see that $\lambda_i=0$ implies $u_i^\top  \mB^{1/2}(x_0-x_*) = 0$. Using this in \eqref{eq:98g9dgisss} gives \eqref{eq:s98g98gjshis}. Finally, inequality \eqref{eq:ug98fg98bdd} follows from
\begin{eqnarray*} 
\|\Exp{x_k-x_*}\|_{\mB}^2 &\overset{\eqref{eq:h98g98ishihfp}}{=}& \sum_{i=1}^n (1-\omega \lambda_i)^{2k} \left(u_i^\top  \mB^{1/2}(x_0-x_*)\right)^2 \\
& =& \sum_{i: \lambda_i>0} (1-\omega \lambda_i)^{2k} \left( u_i^\top  \mB^{1/2}(x_0-x_*)\right)^2   \\
&\overset{\eqref{eq:rho}}{\leq} & \rho^k(\omega)\sum_{i: \lambda_i>0}  \left( u_i^\top  \mB^{1/2}(x_0-x_*)\right)^2  \\
&=& \rho^k(\omega)\sum_{i: \lambda_i>0}  \left( u_i^\top  \mB^{1/2}(x_0-x_*)\right)^2 + \rho^k(\omega) \sum_{i: \lambda_i=0}  \left( u_i^\top  \mB^{1/2}(x_0-x_*)\right)^2 \\
&=&\rho^k(\omega)\sum_{i}  \left(u_i^\top  \mB^{1/2}(x_0-x_*)\right)^2 \\
&=& \rho^k(\omega)\sum_{i}  (x_0-x_*)^\top  \mB^{1/2}u_i  u_i^\top  \mB^{1/2}(x_0-x_*)  \\
&=&\rho^k(\omega) (x_0-x_*)^\top  \mB^{1/2}\left(\sum_i u_i u_i^\top \right)\mB^{1/2}(x_0-x_*)\\
&= & \rho^k(\omega) \|x_0-x_*\|_{\mB}^2.
\end{eqnarray*}
The last identity follows from the fact that $\sum_i u_i u_i^\top = \mU \mU^\top = \mI$.

\subsection{Choice of the stepsize  parameter}
\label{subsec: omega}

We now consider the problem of choosing the stepsize (relaxation) parameter $\omega$. In view of \eqref{eq:ug98fg98bdd} and \eqref{eq:rho}, the optimal relaxation parameter is the one solving the following optimization problem:
\begin{equation} \label{eq:omegamin}
\min_{\omega\in \R} \left\{\rho(\omega) = \max_{i: \lambda_i>0} (1-\omega \lambda_i)^2\right\}.
\end{equation}
In the next result we solve the above problem. 

\begin{theorem}[Stepsize]\label{thm:alg1_bestomega}The objective of~\eqref{eq:omegamin} is given by
\begin{equation}\label{eq:rho-cases}
\rho(\omega) =
\begin{cases}
(1 - \omega\lambda_{\max})^2 & \text{if} \quad \omega \leq 0 \\
(1 - \omega\lambda_{\min}^+)^2 & \text{if} \quad 0 \leq \omega \leq \omega^*\\
(1-\omega\lambda_{\max})^2 & \text{if} \quad \omega \geq \omega^*
\end{cases},
\end{equation}
where $\omega^* \eqdef 2/(\lambda_{\min}^+ + \lambda_{\max} )$. Moreover, $\rho$ is decreasing on $(-\infty, \omega^* ]$ and increasing on
$[\omega^* , +\infty)$, and hence the optimal solution of~\eqref{eq:omegamin} is $\omega^*$. Further, we have:
\begin{enumerate}
\item[(i)] If we choose $\omega = 1$ (no over-relaxation), then
\begin{equation}\label{eq:rho(1)} \rho(1) = (1 - \lambda_{\min}^+)^2. \end{equation}
\item[(ii)] If we choose $\omega = 1/\lambda_{\max}$ (over-relaxation), then
\begin{equation}\label{eq:rho2}\rho(1/\lambda_{\max} ) = \left(1 -\frac{\lambda_{\min}^+}{ \lambda_{\max}}\right)^2  \overset{\eqref{eq:kappa}}{=} \left(1  - \frac{1}{\condnum}\right)^2.
\end{equation}
\item[(iii)]  If we choose $\omega = \omega^*$ (optimal over-relaxation), then the optimal rate is
\begin{equation}\label{eq:rho_optimal}\rho(\omega^* ) = \left(1 -\frac{2\lambda_{\min}^+}{\lambda_{\min}^+ + \lambda_{\max}}\right)^2 = \left(1 - \frac{2}{\condnum + 1}\right)^2.\end{equation}
\end{enumerate}
\end{theorem}
\begin{proof}
Recall that $\lambda_{\max} \leq 1$. 
Letting $\rho_i (\omega) = (1 - \omega\lambda_i)^2$, 
it is easy to see that $\rho(\omega) =\max\{\rho_j (\omega), \rho_n (\omega)\}$, where $j$ is such that $\lambda_j = \lambda_{\min}^+$. Note that $\rho_j(\omega) =\rho_n (\omega)$ for $\omega \in \{0, \omega^* \}$. From this we deduce that $\rho_j\geq \rho_n$ on $(-\infty, 0]$, $\rho_j\leq \rho_n$ on $[0, \omega^* ]$, and
$\rho_j\geq \rho_n$ on $[\omega^* , +\infty)$, obtaining \eqref{eq:rho-cases}. We see that $\rho$ is decreasing on $(-\infty, \omega^* ]$, and increasing on
$[\omega^* , +\infty)$. The remaining results follow directly by plugging specific values of $\omega$ into \eqref{eq:rho-cases}.
\end{proof}

 Theorem~\ref{thm:alg1_bestomega} can be intuitively understood in the following way. By design, we know that $\lambda_{\max} \leq 1$. If we do not have a better bound on the largest eigenvalue, we can simply choose
$\omega = 1$ to ensure convergence. If we have a stronger bound available, say $\lambda_{\max} \leq U < 1$, we can pick $\omega = 1/U$, and the convergence rate will improve. The better the bound, the better
the rate. However, using a stepsize of the form $\omega = 1/U$ where $U$ is not an upper bound on $\lambda_{\max}$ is risky: if we underestimate the eigenvalue by a factor of 2 or more, we can not guarantee convergence. Indeed, if $U \leq \lambda_{\max} /2$, then $1/U \geq 2/\lambda_{\max}$ and hence $\rho(\omega) \geq 1$. Beyond this
point, information about $\lambda_{\min}^+$ is useful. However, the best possible improvement beyond this only leads to a further factor of 2 speedup in terms of the number of iterations. Therefore, one needs to be careful about underestimating $\lambda_{\max}$.

\begin{example}[Random vectors]
An important class of methods is obtained by restricting $\mS$ to random vectors. In this case,
\begin{eqnarray*}\lambda_{\min}^+ +\lambda_{\max} &\leq &   \sum_{i=1}^n\lambda_i =  \trace{\mB^{-1/2} \Exp{\mZ}\mB^{-1/2} } 
= \Exp{\trace{\mB^{-1/2} \mZ\mB^{-1/2} }} \\
&=& \Exp{\trace{\mB^{-1} \mZ }} = \Exp{{\rm dim}(\range{\mB^{-1} \mZ})} = 1,\end{eqnarray*}
and thus
$\omega^* = 2/(\lambda_{\min}^+ + \lambda_{\max}) \geq 2.$ This means that in this case we can always safely choose the relaxation parameter to be $\omega=2$. This  results in faster rate than the choice $\omega=1$.
\end{example}

\subsection{L2 convergence}

In this section we establish a bound on $\Exp{\|x_k-x_*\|_{\mB}^2}$, i.e., we prove $L2$ convergence. This is a stronger type of convergence than what we get by bounding $\|\Exp{x_k-x_*}\|_{\mB}^2$. Indeed, for any random vector $x_k$  we have the inequality (see Lemma~4.1 in \cite{SIMAX2015})
\[\Exp{\|x_k-x_*\|_{\mB}^2}  = \left\|\Exp{x_k-x_*}\right\|_{\mB}^2 + \Exp{\left\| x_k - \Exp{x_k}\right\|^2_{\mB}}.\] 
Hence, L2 convergence also implies that the quantity $\Exp{\left\| x_k - \Exp{x_k}\right\|^2_{\mB}}$---the {\em total variance}\footnote{Total variance of a random vector is the trace of its covariance matrix.} of $x_k$---converges to zero.

We shall first establish an insightful lemma. The lemma connects two important measures of success:   $\|x_k-x_*\|_\mB^2$ and $f(x_k)$.

\begin{lemma} \label{lem:identitiesxx} Choose $x_0\in \R^n$ and let $\{x_k\}_{k=0}^\infty$ be the random iterates produced by Algorithm~\ref{alg:alg1}, with an arbitrary relaxation parameter $\omega\in \R$. Let $x_* \in \cL$. Then we have the identities $\norm{x_{k+1}-x_{k}}_\mB^2 = 2\omega^2 f_{\mS_k}(x_k)$, and
\begin{equation}\label{eq:ho097sjhs78s}
\norm{x_{k+1}-x_*}_\mB^2 = \norm{x_{k}-x_*}_\mB^2 -  2\omega(2- \omega)f_{\mS_k}(x_k).
\end{equation}
Moreover, $\Exp{\|x_{k+1}-x_k\|_{\mB}^2} = 2\omega^2 \Exp{f(x_k)}$, and
\begin{equation}\label{eq:hoihoihsvjhvnmx}
\Exp{\|x_{k+1}-x_*\|_{\mB}^2} = \Exp{\|x_k-x_*\|_{\mB}^2} -  2\omega(2- \omega)\Exp{f(x_k)}.
\end{equation}
\end{lemma}

\begin{proof}
Recall that Algorithm~\ref{alg:alg1} performs the update $x_{k+1} = x_k - \omega \mB^{-1} \mZ_k (x_k-x_*)$. From this we get
\begin{equation}\label{eq:ohihosytrsl} \|x_{k+1}-x_k\|_{\mB}^2  \overset{\eqref{eq:ZBZ}}{=} \omega^2  (x_k-x_*)^\top \mZ_k (x_k-x_*)  \overset{\eqref{eq:prodstoch2}}{=} 2\omega^2 f_{\mS_k}(x_k), \end{equation}
In a similar vein,
\begin{eqnarray*}
\norm{x_{k+1}-x_*}_{\mB}^2 &= &\norm{(\mI - \omega\mB^{-1} \mZ_k)(x_k-x_*)}_{\mB}^2\\
& =& (x_k-x_*)^\top (\mI - \omega \mZ_k\mB^{-1} )\mB(\mI - \omega\mB^{-1} \mZ_k)(x_k-x_*)\\
&\overset{\eqref{eq:ZBZ}}{=} & (x_k-x_*)^\top (\mB - \omega(2-\omega) \mZ_k)(x_k-x_*)\\
&\overset{\eqref{eq:prodstoch2}}{=} & \norm{x_k-x_*}_{\mB}^2 -  2 \omega(2- \omega) f_{\mS_k}(x_k),
\end{eqnarray*}
establishing  \eqref{eq:ho097sjhs78s}. Taking expectation in \eqref{eq:ohihosytrsl}, we get
\begin{eqnarray*}\Exp{\norm{x_{k+1}-x_{k}}_\mB^2} = \Exp{\Exp{\norm{x_{k+1}-x_{k}}_\mB^2 \;|\; x_k} } = 2\omega^2 \Exp{\Exp{f_{\mS_k}(x_k) \;|\; x_k}} = 2\omega^2 \Exp{f(x_k)}.\end{eqnarray*}
Taking expectation in  \eqref{eq:ho097sjhs78s}, we get $\Exp{\norm{x_{k+1}-x_*}_{\mB}^2 \, | \, x_k} =  \norm{x_k-x_*}_{\mB}^2 -  2 \omega(2- \omega) f(x_k)$. It remains to take expectation again.
\end{proof}

In our next result we utilize Lemma~\ref{lem:identitiesxx}  to establish L2 convergence of the basic method. 
 
\begin{theorem}[$L2$ convergence]\label{thm:alg1_complexity_second} Let Assumption~\ref{ass:main} hold and set $x_* = \Pi^{\mB}_{\cL}(x_0)$. Let $\{x_k\}$ be the random iterates produced by Algorithm~\ref{alg:alg1}, where the relaxation parameter satisfies $0<\omega <2$. 

\begin{itemize}
\item[(i)]
For $k\geq 0$ we have
\begin{eqnarray}(1-\omega(2-\omega)\lambda_{\max})^k \|x_0-x_*\|_{\mB}^2 & \leq & \Exp{\|x_k-x_*\|_{\mB}^2} \notag \\
&\leq & (1-\omega(2-\omega)\lambda_{\min}^+)^k \|x_0-x_*\|_{\mB}^2. \phantom{aabcd}
\label{eq:hhshgygrrreeeo}
\end{eqnarray}

\item[(ii)] The average iterate $\hat{x}_k \eqdef \frac{1}{k}\sum_{t=0}^{k-1} x_t$ for  all $k\geq 1$ satisfies
\begin{equation}\label{eq:kjhkjhuys9s}\Exp{\|\hat{x}_k - x_*\|_{\mB}^2} \leq \frac{\|x_0-x_*\|_{\mB}^2}{2\omega(2-\omega)\lambda_{\min}^+ k}.\end{equation}
\end{itemize}

The best rate  is achieved when $\omega =1.$


\end{theorem}

\begin{proof} Let $\phi_k = \Exp{f(x_k)}$ and $r_k = \Exp{\|x_k-x_*\|_{\mB}^2}$. 

\begin{enumerate}
\item[(i)]  We have
$r_{k+1} \overset{\eqref{eq:hoihoihsvjhvnmx}}{=} r_k - 2\omega (2-\omega)\phi_k \overset{\eqref{eq:jsguygusopoipoip} }{\leq} r_k - \omega (2-\omega)\lambda_{\min}^+ r_k,$
and
$r_{k+1} \overset{\eqref{eq:hoihoihsvjhvnmx}}{=} r_k - 2\omega (2-\omega)\phi_k \overset{\eqref{eq:jshvs5r6989ss} }{\geq} r_k - \omega(2-\omega)\lambda_{\max} r_k.$
Inequalities \eqref{eq:hhshgygrrreeeo} follow from this by unrolling the recurrences.

\item[(ii)] By summing up  the identities from \eqref{eq:hoihoihsvjhvnmx}, we get
$2\omega (2-\omega)\sum_{t=0}^{k-1} \phi_t = r_0 - r_k$. Therefore,
\begin{eqnarray*}\Exp{\|\hat{x}_k - x_*\|_{\mB}^2}& =& \Exp{\left\|\frac{1}{k}\sum_{t=0}^{k-1}( x_t - x_*) \right\|_{\mB}^2} \leq \Exp{  \frac{1}{k}\sum_{t=0}^{k-1} \left\| x_t - x_* \right\|_{\mB}^2  } \\
&=&\frac{1}{k}\sum_{t=0}^{k-1} r_t \overset{\eqref{eq:jsguygusopoipoip}}{\leq} \frac{1}{\lambda_{\min}^+ k}\sum_{t=0}^{k-1} \phi_t   \leq \frac{r_0}{2\omega(2-\omega)\lambda_{\min}^+ k}.
\end{eqnarray*}

\end{enumerate}

\end{proof}

Note that in part (i) we give both an upper and a {\em lower} bound on $\Exp{\|x_k-x_*\|_\mB^2}$.  

\subsection{Convergence of expected function values}

In this section we establish a linear convergence rate for the decay of $\Exp{f(x_k)}$ to zero. We prove two results, with different quantitative (speed) and qualitative (assumptions and insights gained) qualities. 

The complexity in the first result (Theorem~\ref{thm:complexity_f} ) is disappointing: it is (slightly) worse than quadratic in the condition number $\condnum$. However, we do not need to invoke Assumption~\ref{ass:main} (exactness). In addition, this result implies that the expected function values decay {\em monotonically} to zero.

\begin{theorem}[Convergence of expected function values] \label{thm:complexity_f} Choose any $x_0\in \R^n$ and let
$\{x_k\}$ be the random iterates produced by Algorithm~\ref{alg:alg1}, where $0\leq \omega\leq 2/\condnum$ (note that $2/\condnum =  2\lambda_{\min}^+/\lambda_{\max} \leq 2$). Then ${\rm E} \left[f(x_{k+1})\right] \leq (1- 2 \lambda_{\min}^+ \omega + \lambda_{\max} \omega^2) {\rm E}\left[f(x_k)\right]$  for all $k$, and hence
\begin{equation}\label{eq:098ys0huhd8hddd}\Exp{f(x_{k})} \leq (1  - 2 \lambda_{\min}^+\omega +\lambda_{\max} \omega^2)^k f(x_0).
\end{equation}
The optimal rate is achieved for $\omega = 1/\condnum$, in which case we get the bound
$\Exp{f(x_k)} \leq \left(1 - \frac{(\lambda_{\min}^+)^2}{\lambda_{\max}}\right)^k f(x_0).$
\end{theorem}
\begin{proof}
Let $\mS\sim \cD$ be independent  from $\mS_0,\mS_1,\dots,\mS_k$ and fix any $x_*\in \cL$. Then we have
\begin{eqnarray*}f(x_{k+1}) &\overset{\eqref{eq:min_f}}{=}& \EE{\mS\sim \cD}{f_{\mS}(x_{k+1})}\\
&\overset{\eqref{eq:alg:SGD}}{=}& \EE{\mS \sim \cD}{f_{\mS} (x_k - \omega \nabla f_{\mS_k}(x_k))}\\
&\overset{\eqref{eq:prodstoch2} }{=}&   \frac{1}{2}\EE{\mS \sim \cD}{(x_k - x_* - \omega \nabla f_{\mS_k}(x_k))^\top \mZ (x_k - x_* - \omega \nabla f_{\mS_k}(x_k))}\\
&=&\frac{1}{2}(x_k - x_*-\omega \nabla f_{\mS_k}(x_k))^\top \Exp{\mZ} (x_k - x_*-\omega \nabla f_{\mS_k}(x_k))\\
&=& \frac{1}{2}(x_k - x_*)^\top \Exp{\mZ} (x_k - x_*) - \omega (\nabla f_{\mS_k}(x_k))^\top \Exp{\mZ} (x_k-x_*) \\
&& \quad + \frac{\omega^2}{2}\|\nabla f_{\mS_k}(x_k)\|_{\Exp{\mZ}}^2\\
&\overset{\eqref{eq:probminf2}}{=}& f(x_k) - \omega (\nabla f_{\mS_k}(x_k))^\top \Exp{\mZ} (x_k-x_*) + \frac{\omega^2}{2}\|\nabla f_{\mS_k}(x_k)\|_{\Exp{\mZ}}^2.
\end{eqnarray*}

Taking expectations, conditioned on $x_k$ (that is, the expectation is with respect to $\mS_k$), we can further write
\begin{equation}\label{eq:hbd87g7udds}\Exp{f(x_{k+1}) \;|\; x_k} =  f(x_k) - \omega \alpha_k + \omega^2 \beta_k, \end{equation}
where
\begin{equation}\label{eq:alpha_nbf97g9fg} \alpha_k \eqdef (\EE{\mS_k\sim \cD}{\nabla f_{\mS_k}(x_k)})^\top \Exp{\mZ} (x_k-x_*),\end{equation}
and
\begin{equation}\label{eq:beta_bg9f87gf} \beta_k\eqdef \frac{1}{2}\EE{\mS_k\sim \cD}{\|\nabla f_{\mS_k}(x_k)\|_{\Exp{\mZ}}^2}. \end{equation}
We shall now bound $\alpha_k$ from below and $\beta_k$ from above in terms of $f(x_k)$. Using the inequality $\Exp{\mZ}\preceq \lambda_{\max} \mB$ (this follows from 
 $\mB^{-1/2}\Exp{\mZ}\mB^{-1/2}\preceq \lambda_{\max}\mI$), we get
\[\beta_k \overset{\eqref{eq:beta_bg9f87gf} }{\leq} \frac{\lambda_{\max}}{2} \EE{\mS_k\sim \cD}{\|\nabla f_{\mS_k}(x_k)\|_{\mB}^2} \overset{\eqref{eq:v(x)}}{=} \lambda_{\max}  \EE{\mS_k\sim \cD}{f_{\mS_k}(x_k)} \overset{\eqref{eq:min_f}}{=} \lambda_{\max}f(x_k).\]

On the other hand,
\[
\alpha_k \overset{\eqref{eq:alpha_nbf97g9fg} +\eqref{eq:grad_f}}{=} (x_k-x_*)^\top \Exp{\mZ} \mB^{-1} \Exp{\mZ} (x_k-x_*)
\overset{\eqref{eq:grad_f}}{=}\|\nabla f(x_k)\|_{\mB}^2
\overset{\eqref{eq:sandwich}}{\geq} 2 \lambda_{\min}^+ f(x_k).
\]

Substituting the bounds for $\alpha_k$ and $\beta_k$ into \eqref{eq:hbd87g7udds}, we get  
$\Exp{f(x_{k+1}) \;|\; x_k} \leq (1- 2\lambda_{\min}^+ \omega +  \lambda_{\max} \omega^2)f(x_k)$. Taking expectations again gives
\[ \Exp{f(x_{k+1})} = \Exp{\Exp{f(x_{k+1}) \;|\; x_k}}
\leq (1 - 2\lambda_{\min}^+ \omega + \lambda_{\max}\omega^2) \Exp{f(x_k)}.
\]
It remains to unroll the recurrence.
\end{proof}

We now present an alternative convergence result (Theorem~\ref{thm:complexity_f_2}), one in which we do not bound the decrease in terms of the  initial function value, $f(x_0)$, but in terms of a somewhat larger quantity. This allows us to provide a better convergence rate. For this result to hold, however, we need to invoke Assumption~\ref{ass:main}. Note also that this result does not imply that the expected function values decay monotonically.

\begin{theorem}[Convergence of expected function values]\label{thm:complexity_f_2} 
 Choose $x_0\in\R^n$, and let $\{x_k\}_{k=0}^\infty$ be the random iterates produced by Algorithm~\ref{alg:alg1}, where the relaxation parameter satisfies $0<\omega <2$.

\begin{itemize}
 \item[(i)] Let $x_*\in \cL$. The average iterate $\hat{x}_k \eqdef \frac{1}{k}\sum_{t=0}^{k-1} x_t$ for  all $k\geq 1$ satisfies
\begin{equation}\label{eq:kjhkjhuys9sxxss}\Exp{f(\hat{x}_k)} \leq \frac{\|x_0-x_*\|_{\mB}^2}{2\omega(2-\omega) k}.\end{equation}
 \item[(ii)] Now let Assumption~\ref{ass:main} hold. For $x_* = \Pi^{\mB}_{\cL}(x_0)$ and $k\geq 0$ we have
\begin{equation}
\Exp{f(x_k)} \leq \left(1- \omega(2- \omega)\lambda_{\min}^+ \right)^k \frac{\lambda_{\max}\norm{x_0-x_*}_{\mB}^2}{2}.
\end{equation}
\end{itemize}
The best rate  is achieved when $\omega =1.$
\end{theorem}
\begin{proof} 
(i) Let $\phi_k = \Exp{f(x_k)}$ and $r_k = \Exp{\|x_k-x_*\|_{\mB}^2}$. By summing up  the identities from \eqref{eq:hoihoihsvjhvnmx}, we get
$2\omega (2-\omega)\sum_{t=0}^{k-1} \phi_t = r_0 - r_k$. Therefore, using Jensen's inequality,
\begin{eqnarray*}\Exp{f(\hat{x}_k)}& \leq & \Exp{\frac{1}{k}\sum_{t=0}^{k-1} f(x_t) } =   \frac{1}{k}\sum_{t=0}^{k-1} \phi_t = \frac{r_0-r_k}{2\omega(2-\omega) k} \leq \frac{r_0}{2\omega(2-\omega) k}.
\end{eqnarray*}

(ii) Combining inequality \eqref{eq:jshvs5r6989ss} with Theorem~\ref{thm:alg1_complexity_second}, we get
\[\Exp{f(x_k)} \leq \frac{\lambda_{\max}}{2}\Exp{ \|x_k-x_*\|_{\mB}^2} \overset{\eqref{eq:hhshgygrrreeeo}}{\leq}   \left(1- \omega(2- \omega)\lambda_{\min}^+ \right)^k \frac{\lambda_{\max}\norm{x_0-x_*}_{\mB}^2}{2}. \]

\end{proof}

\begin{remark} Theorems~\ref{thm:complexity_f} and \ref{thm:complexity_f_2} are complementary. In particular, the complexity result given in Theorem~\ref{thm:complexity_f} (for the last iterate) holds under weaker assumptions. Moreover, Theorem~\ref{thm:complexity_f}  implies monotonicity of expected function values. On the other hand, the rate is substantially better in Theorem~\ref{thm:complexity_f_2}. Also, Theorem~\ref{thm:complexity_f_2} applies to a wider range of stepsizes.
\end{remark}

It is also possible to obtain other convergence results as a corollary. For instance,  one can get a linear rate for the decay of the norms of the gradients as a corollary of Theorems~\ref{thm:complexity_f} and \ref{thm:complexity_f_2} using the upper bound in Lemma~\ref{lem:sandwich}.

\section{Parallel and Accelerated Methods} \label{sec:par_and_acc}

In this section we propose and analyze  {\em parallel} and {\em accelerated} variants of Algorithm~\ref{alg:alg1}.

\subsection{Parallel method} \label{sec:minibatch}

The parallel method \eqref{eq:parallel_method_intro} is formalized in this section as Algorithm~\ref{alg:alg2}.

\begin{algorithm}[!h]
\begin{algorithmic}[1]
\State \textbf{Parameters:} distribution $\cD$ from which to sample matrices; positive definite matrix $\mB\in \R^{n \times n}$; stepsize/relaxation parameter $\omega\in \R$; parallelism parameter $\mb$
\State Choose $x_0 \in \R^n$
\Comment Initialization
\For {$k = 0, 1, 2, \dots$}
	\For {$i = 1,2,\dots,\mb$}
	\State Draw $\mS_{ki} \sim \cD$
	\State Set $z_{k+1, i} = x_k -  \omega \mB^{-1} \mA^\top  \mS_{ki} (\mS_{ki}^\top  \mA \mB^{-1} \mA^\top  \mS_{ki})^\dagger \mS_{ki}^\top  (\mA x_k - b) $
\EndFor	
	\State 	Set $x_{k+1} = \frac{1}{\mb}\sum_{i=1}^{\mb} z_{k+1,i}$	
	\Comment	 Average the results
\EndFor
\end{algorithmic}
\caption{Parallel Method}
\label{alg:alg2}
\end{algorithm}

For brevity,  we only prove L2 convergence results. However, various other results can be obtained as well, as was the case for the basic method, such as convergence of expected iterates, expected function values and average iterates.

\begin{theorem} \label{thm:minibatch} Let Assumption~\ref{ass:main} hold and set $x_* = \Pi^{\mB}_{\cL}(x_0)$. Let $\{x_k\}_{k=0}^\infty$ be the random iterates produced by Algorithm~\ref{alg:alg2}, where the relaxation parameter satisfies $0 < \omega < 2/\xi(\mb)$, and $\xi(\mb)\eqdef \frac{1}{\mb}   + \left(1 - \frac{1}{\mb}\right) \lambda_{\max}$. Then 
\begin{equation}\label{eq:bu9fgffssd}\Exp{\norm{x_{k+1}-x_*}_{\mB}^2} \leq \rho(\omega,\mb)\cdot \Exp{\norm{x_{k}-x_*}_{\mB}^2},\end{equation}
and
$\Exp{f(x_k)} \leq \rho(\omega,\mb)^k \frac{\lambda_{\max}}{2}\|x_0-x_*\|_{\mB}^2,$
where $\rho(\omega,\mb) \eqdef  1 - \omega\left[2  - \omega \xi(\mb)\right] \lambda_{\min}^+ $.
For any fixed $\mb\geq 1$, the optimal stepsize choice is $\omega(\mb) \eqdef 1/\xi(\mb)$ and the associated optimal rate is \begin{equation}\label{eq:optimal_mb_rate}\rho(\omega(\mb),\mb) = 1 - \frac{\lambda_{\min}^+}{\tfrac{1}{\mb} + \left(1-\tfrac{1}{\mb}\right)\lambda_{\max}} .\end{equation}
In particular, if we use the optimal stepsize and let $K(\tau) \eqdef \left(1-\frac{1}{\tau}\right) \frac{\lambda_{\max}}{\lambda_{\min}^+} + \frac{1}{\tau}\frac{1}{\lambda_{\min}^+} $, then \eqref{eq:bu9fgffssd} implies  
\begin{equation}\label{eq:b97g98dbjjgsgdx} k \geq K(\tau) \log \frac{1}{\epsilon}\quad  \Rightarrow  \quad \Exp{\|x_k-x_*\|_{\mB}^2} \leq \epsilon \|x_0-x_*\|_{\mB}^2,\end{equation}

\end{theorem}

\begin{proof}
Recall that Algorithm~\ref{alg:alg2} performs the update $x_{k+1} = x_k - \omega  \mB^{-1} \tilde{\mZ}_{k} (x_k-x_*)$, where $\tilde{\mZ}_k \eqdef \frac{1}{\mb}\sum_{i=1}^{\mb} \mZ_{ki}$. We have
\begin{eqnarray}
&& \Exp{\norm{x_{k+1}-x_*}_{\mB}^2  \;|\; x_k} \notag \\
&=& \Exp{\norm{(\mI - \omega\mB^{-1} \tilde{\mZ}_k)(x_k-x_*)}_{\mB}^2}\notag\\
& =& \Exp{(x_k-x_*)^\top (\mI - \omega \tilde{\mZ}_k\mB^{-1} )\mB(\mI - \omega\mB^{-1} \tilde{\mZ}_k)(x_k-x_*)}\notag\\
&\overset{\eqref{eq:ZBZ}}{=} & \Exp{(x_k-x_*)^\top (\mB - 2\omega \tilde{\mZ}_k + \omega^2 \tilde{\mZ}_k \mB^{-1} \tilde{\mZ}_k (x_k-x_*)}\notag\\
&=&(x_k-x_*)^\top \left(\mB - 2\omega \Exp{\mZ} + \omega^2 \Exp{\tilde{\mZ}_k \mB^{-1} \tilde{\mZ}_k} \right)(x_k-x_*).\label{eq:ihs8y98y98ys}
\end{eqnarray}
Next, we can write
$
\tilde{\mZ}_k \mB^{-1} \tilde{\mZ}_k = \frac{1}{\mb^2} \left(\sum_{i=1}^{\mb} \mZ_{ki} \mB^{-1} \mZ_{ki} + \sum_{(i,j)\;:\; i\neq j} \mZ_{ki} \mB^{-1} \mZ_{kj}\right).
$
Since $\mZ_{ki} \mB^{-1} \mZ_{ki} = \mZ_{ki}$, and because $\mZ_{ki}$ and $\mZ_{kj}$ are independent for $i\neq j$, we have
\begin{eqnarray}\Exp{ \tilde{\mZ}_k \mB^{-1} \tilde{\mZ}_k } &=& \frac{1}{\mb^2} \left( \mb \Exp{\mZ} +  (\mb^2 - \mb) \Exp{\mZ} \mB^{-1} \Exp{\mZ}\right) \notag \\
& \preceq & \left(\frac{1}{\mb}   + \left(1 - \frac{1}{\mb}\right) \lambda_{\max}\right) \Exp{\mZ},\label{eq:iugs7334}
\end{eqnarray}
where we have used the estimate $\Exp{\mZ} \mB^{-1} \Exp{\mZ} \preceq \lambda_{\max} \Exp{\mZ}$, which follows from the bound $\mW^2 \leq \lambda_{\max} \mW$. Plugging \eqref{eq:iugs7334} into \eqref{eq:ihs8y98y98ys}, and noting that $\|x_k-x_*\|_{\Exp{\mZ}}^2 = 2 f(x_k)$, we obtain:
\begin{eqnarray*}
\Exp{\norm{x_{k+1}-x_*}_{\mB}^2\;|\; x_k} &\leq &
\|x_k-x_*\|_{\mB}^2 - \left[2\omega  - \omega^2 \left(\frac{1}{\mb}   + \left(1 - \frac{1}{\mb}\right) \lambda_{\max}\right)\right] 2 f(x_k)\\
& \overset{\eqref{eq:jsguygusopoipoip} }{\leq}& \rho(\omega,\mb)\|x_k-x_*\|_{\mB}^2 .
\end{eqnarray*}
The inequality involving $f$ is shown in the same way as in Theorem~\ref{thm:complexity_f_2}.
\end{proof}

Since $\lambda_{\max}\leq 1$,  $K$ is a non-increasing function of $\tau$, with $K(1) = 1/\lambda_{\min}^+$ and $K(\infty)\eqdef \lim_{\tau\to \infty} K(\tau) = \lambda_{\max}/\lambda_{\min}^+ $. In the asymptotic regime $\tau \to \infty$, Algorithm~\ref{alg:alg2} becomes gradient descent for minimizing $f$, and $K(\infty)$ is the standard rate of gradient descent.
The quantity $\frac{K(1)}{K(\infty)}=\frac{1}{\lambda_{\max}}$ controls the maximum (guaranteed) speedup in the iteration complexity achievable by increasing $\tau$.  

Note that for $\tau \geq 1/\lambda_{\max}$, we get $K(\tau) \leq (2-\lambda_{\max}) K(\infty) \leq 2 K(\infty)$, which is the performance of gradient descent (up to a factor of $2$). This means that it does not make sense to use a minibatch size  larger than $1/\lambda_{\max}$.

Further, notice that $K(\tau) \geq \frac{1}{\tau} K(1)$ for all $\tau$. This means that the number of iterations does not decrease linearly in the minibatch size  $\tau$. 

This also means that in a computational regime where processing $\tau$ basic method updates costs $\tau$ times as much as processing a single update, the decrease in iteration complexity can't compensate for the increase in cost per iteration, which means that the choice $\tau=1$ is optimal. On the other hand, if a parallel processor is available, a larger $\tau$ will be optimal.


\subsection{Accelerated method}\label{sec:acceleration}

In this section we develop an accelerated variant of Algorithm~\ref{alg:alg1}.  Recall that a single iteration of Algorithm~\ref{alg:alg1} takes the form $x_{k+1} = \phi_{\omega}(x_k,\mS_k)$, 
where \begin{equation}\label{eq:phi-acc}\phi_{\omega}(x,\mS)\eqdef x - \omega\mB^{-1}\mA^\top \mS (\mS^\top \mA \mB^{-1} \mA^\top \mS)^\dagger\mS^\top (\mA x - b).\end{equation}
We have seen that the convergence rate progressively improves as we increase $\omega$ from $1$ to  $ \omega_*$, which is the optimal choice. In particular, with $\omega=1$ we have the complexity  $\tilde{\cO}(1/\lambda_{\min}^+)$, while choosing $\omega=1/\lambda_{\max} = 1/\lambda_{\max}$ or $\omega = \omega_*$ leads to the improved complexity  $\tilde{\cO}(\lambda_{\max}/\lambda_{\min}^+) = \tilde{\cO}(\condnum)$.

In order to obtain further acceleration, we suggest to perform an update step in which $x_{k+1}$ depends on both $x_{k}$ and $x_{k-1}$. In particular, we take two {\em dependent} steps of Algorithm~\ref{alg:alg1}, one from $x_k$ and one from $x_{k-1}$, and take an affine combination of the results. This, the process is started with $x_0,x_1\in \R^n$, and for $k\geq 1$ involves an iteration of the form
\[x_{k+1} = \gamma \phi_{\omega}(x_k,\mS_k) + (1-\gamma)\phi_{\omega}(x_{k-1},\mS_{k-1}),\]
where  the  matrices $\{\mS_k\}$ are  independent samples from $\cD$, and $\gamma\in \R$ is an {\em acceleration parameter}. Note that by choosing $\gamma=1$ (no acceleration), we recover Algorithm~\ref{alg:alg1}. This method is formalized as Algorithm~\ref{alg:alg3}.

\begin{algorithm}[!h]
\begin{algorithmic}[1]
\State \textbf{Parameters:} distribution $\cD$ from which to sample matrices; positive definite matrix $\mB\in \R^{n \times n}$; stepsize/relaxation parameter $\omega>0$; acceleration parameter $\gamma>0$
\State Choose $x_0, x_1\in \R^n$ such that $x_0-x_1\in \range{\mB^{-1}\mA^\top}$ (for instance, choose $x_0=x_1$)
\State Draw $\mS_0\sim \cD$
\State Set $z_0 = \phi_{\omega}(x_0,\mS_0)$ 
\For {$k = 1, 2, \dots$}
	\State Draw a fresh sample $\mS_k \sim \cD$
	\State Set $z_{k} = \phi_{\omega}(x_k,\mS_k)$
	\State Set $x_{k+1} = \gamma z_k + (1-\gamma) z_{k-1}$ 		 \Comment Main update step
\EndFor
\State Output $x_k$
\end{algorithmic}
\caption{Accelerated Method}
\label{alg:alg3}
\end{algorithm}

As we shall see, by a proper combination of overrelaxation (choice of $\omega$) with acceleration (choice of $\gamma$), Algorithm~\ref{alg:alg3} enjoys the accelerated complexity of $\tilde{\cO}(\sqrt{\condnum})$.  We start with a lemma describing the evolution of the expected iterates.

\begin{lemma}[Expected iterates]\label{lem:accel1} Let $x_*$ be any solution of $\mA x = b$ and let $r_k\eqdef \Exp{x_k-x_*}$. Then for all $k$ we have the recursion
\begin{equation}\label{eq:rec1} r_{k+1} = \gamma (\mI-\omega \mB^{-1}\Exp{\mZ}) r_k + (1-\gamma) (\mI- \omega \mB^{-1}\Exp{\mZ}) r_{k-1}. \end{equation}
\end{lemma}

\begin{proof} By taking expectations on both sides of $x_{k+1} = \gamma z_k + (1-\gamma) z_{k-1}$, we get $\Exp{x_{k+1}} = \gamma \Exp{\phi_{\omega}(x_k,\mS_k)} + (1-\gamma)\Exp{\phi_{\omega}(x_{k-1},\mS_{k-1})}.$ After subtracting $x_*$ from both sides, using \eqref{eq:phi-acc}, and replacing $b$ by $\mA x_*$, we get
\begin{eqnarray*}r_{k+1} &=& \gamma \Exp{(\mI - \omega\mB^{-1}\mZ_k) (x_k - x_*)} + (1-\gamma)\Exp{(\mI - \omega\mB^{-1}\mZ_{k-1}) (x_{k-1} - x_*)},
\end{eqnarray*}
where $\mZ_k= \mA^\top \mS_k (\mS_k^\top \mA \mB^{-1} \mA^\top \mS_k)^\dagger\mS_k^\top \mA$. We now use the tower property and linearity of expectation:
\begin{eqnarray*}r_{k+1}
&=&  \gamma \Exp{\Exp{(\mI - \omega\mB^{-1}\mZ_k) (x_k - x_*) \;|\; x_k}} \\
&& \qquad + (1-\gamma)\Exp{\Exp{(\mI - \omega\mB^{-1}\mZ_{k-1}) (x_{k-1} - x_*)\;|\; x_{k-1}}}\\
&=& \gamma \Exp{(\mI - \omega\mB^{-1}\Exp{\mZ}) (x_k - x_*) } + (1-\gamma)\Exp{(\mI - \omega\mB^{-1}\Exp{\mZ}) (x_{k-1} - x_*)}\\
&=&\gamma (\mI - \omega\mB^{-1}\Exp{\mZ}) r_k + (1-\gamma)(\mI - \omega\mB^{-1}\Exp{\mZ}) r_{k-1}.
\end{eqnarray*}
\end{proof}

We can now state our main complexity result. Note that the optimal choice of parameters, covered in case (i), leads to a rate which depends on the square root of the condition number.

\begin{theorem}[Complexity of Algorithm~\ref{alg:alg3}] \label{thm:main-accelerated}
Let Assumption~\ref{ass:main} be satisfied and let $\{ x_k \}_{k=0}^\infty$ be the sequence of random iterates produced by Algorithm~\ref{alg:alg3}, started with $x_0, x_1\in \R^n $ satisfying the relation $x_0-x_1\in \range{\mB^{-1}\mA^\top}$, with relaxation parameter 
$0<\omega \leq 1/\lambda_{\max}$ and acceleration parameter  $\gamma =  2/(1+\sqrt{\mu})$, where $\mu \in (0,\omega\lambda_{\min}^+)$. Let $x_* = \Pi^{\mB}_{\cL}(x_0)$. Then there exists a constant $C>0$, such that for all $k\geq 2$ we have
\begin{equation}\label{eq:rate-accel}\|\Exp{x_k - x_*}\|_{\mB}^2\leq  (1-\sqrt{\mu})^{2k} C.\end{equation}
 \begin{enumerate}
\item[(i)] If we choose $\omega = 1/\lambda_{\max}$ (overrelaxation), then we can pick $\mu =0.99/\condnum$ (recall that $\condnum = \lambda_{\max}/\lambda_{\min}^+$ is the condition number), which leads to the rate
\[\left\|\Exp{x_k - x_*}\right\|_{\mB}^2\leq  \left(1-\sqrt{0.99 \lambda_{\min}^+ /\lambda_{\max} } \right)^{2k} C.\]
\item[(ii)] If we choose $\omega=1$ (no overrelaxation), then we can pick $\mu=0.99 \lambda_{\min}^+$, which leads to the rate
\[\|\Exp{x_k - x_*}\|_{\mB}^2\leq  \left(1-\sqrt{0.99\lambda_{\min}^+}\right)^{2k} C.\]
\end{enumerate}
\end{theorem}


\begin{proof}
Multiplying the identity in Lemma~\ref{lem:accel1} from the left by $\mB^{1/2}$, we obtain
\begin{eqnarray*}
\mB^{1/2} r_{k+1}   &=& \gamma  \left(
\mI   - \omega\mB^{-1/2} \Exp{\mZ} \mB^{-1/2}     
 \right) \mB^{1/2} r_k \\
 &&\quad
  + (1-\gamma)    \left( 
 \mI    - \omega\mB^{-1/2} \Exp{\mZ} \mB^{-1/2}    
   \right) \mB^{1/2} r_{k-1} .
\end{eqnarray*}
Plugging the eigenvalue decomposition $\mU \Lambda \mU^\top$ of  $\mB^{-1/2} \Exp{\mZ} \mB^{-1/2}$ into the above, and multiplying both sides from the left by $\mU^\top$, we get
\begin{equation}
\mU^\top  \mB^{1/2} r_{k+1}  
=  \gamma  \left(
 \mI   -  \omega \Lambda      
 \right) \mU^\top  \mB^{1/2} r_k
  + (1-\gamma)    ( 
  \mI    -  \omega\Lambda     
   ) \mU^\top  \mB^{1/2} r_{k-1} .
   \label{eq:asfdjajflkjsafa} 
\end{equation}

Now if we denote $w_k = \mU^\top  \mB^{1/2} r_k\in \R^n$, \eqref{eq:asfdjajflkjsafa} becomes 
separable in the coordinates of $w$:
\begin{equation}\label{asfsafsafsafa}
 w_{k+1} = \gamma (\mI-\omega\Lambda) w_k 
                   + (1-\gamma) (\mI-\omega\Lambda) w_{k-1}.
\end{equation}

Writing this coordinate-by-coordinate (with $w_k^i$ indicating the $i$th coordinate of $w_k$), we get
\begin{equation}
w_{k+1}^i  
=  \gamma  \left(
 1   -   \omega\lambda_i     
 \right) w_k^i
  + (1-\gamma)     ( 
  1    -  \omega\lambda_i  
   ) w_{k-1}^i, \quad i=1,2,\dots,n.
\label{eq:afdjfjaosfjsafa}
\end{equation}

We now fix $i$ and analyze recursion \eqref{eq:afdjfjaosfjsafa}. We can use  Lemma~\ref{lemma:DS}
with $E_i = \gamma (1-\omega\lambda_i)$ and $F_i = (1-\gamma) (1-\omega\lambda_i)$.  Now recall that  $0\leq\lambda_i\leq 1$ for all $i$, and $\lambda_{\min}^+>0$. Since we assume that $0<\omega < 1/\lambda_{\max}$, we know that $ 0<\omega\lambda_i \leq 1$ for all $i$ for which $\lambda_i>0$, and $\omega \lambda_i = 0$ for those $i$ for which $\lambda_i=0$. Therefore,  it is enough to consider the following 3 cases:
\begin{enumerate}
\item[(1)] $\omega\lambda_i = 1$. In this case we see from \eqref{eq:afdjfjaosfjsafa} that
$w_k^i = 0$ for all $k\geq 2$.

\item[(2)] $\omega\lambda_i = 0$. Since, by assumption, $x_0-x_1\in \range{\mB^{-1}\mA^\top}$, it follows that $\Pi^{\mB}_{\cL}(x_0) = \Pi^{\mB}_{\cL}(x_1)$. All our arguments up to this point hold for arbitrary $x_*\in \cL$. However, we now choose $x_* = \Pi^{\mB}_{\cL}(x_0) = \Pi^{\mB}_{\cL}(x_1)$. Invoking Lemma~\ref{lem:whath_appens_with_zero_eigenvalues} twice, once for $x=x_0$ and then for $x=x_1$, we conclude that $w_0^i=u_i^\top \mB^{1/2}(x_0-x_*) = 0$ and $w_1^i = u_i^\top \mB^{1/2}(x_1-x_*)=0$. In view of recursion \eqref{eq:afdjfjaosfjsafa}, we conclude that $\omega_k^i=0$ for all $k\geq 0$.

\item[(3)] $0<\omega\lambda_i < 1$. In this case we have
\begin{eqnarray*}
E_i^2 + 4F_i &=&
\gamma^2 (1-\omega\lambda_i)^2
 +4(1-\gamma) (1-\omega\lambda_i) 
= (1-\omega\lambda_i)
\left(
(2-\gamma)^2
-
\omega\lambda_i \gamma^2
\right)
\\
&=& (1-\omega\lambda_i)
\left(
\left(\frac{2\sqrt{\mu}}{1+\sqrt{\mu}}\right)^2
-
\omega\lambda_i 
\left(\frac{2}{1+\sqrt{\mu}}\right)^2
\right)
 \\
 &=& 4 \frac{(1-\omega\lambda_i)}{
(1+\sqrt{\mu})^2}
\left(
\mu
-
\omega\lambda_i
\right)
< 0,
\end{eqnarray*}
where the last inequality follows from the assumption $\mu < \omega \lambda_{\min}^+$. Therefore, we can apply Lemma~\ref{lemma:DS}, using which we can deduce the bound
\begin{eqnarray}
 w_{k}^i 
 &= &
 2 M_i^k \left( X_i \cos(\theta k) + Y_i \sin(\theta k)\right) \label{eq:nbifgisg89gu9f}
 \\
 &\leq & 
  2 \left(  \sqrt{    \frac{E_i^2}{4}  +  \frac{  -E_i^2 - 4 F_i}{4}                      }         \right)^k \sqrt{X_i^2 + Y_i^2} \sqrt{\cos^2 (\theta k) + \sin^2 (\theta k) }
\notag\\
&=& 
  2\left(  
  \sqrt{     -  F_i                      }         \right)^k \sqrt{X_i^2 + Y_i^2} 
\quad = \quad
  2 \left(   
  \sqrt{           \frac{1-\sqrt{\mu}}{1+\sqrt{\mu}}
     (1-\omega\lambda_i)                       }         \right)^k \sqrt{X_i^2 + Y_i^2}  \notag  \\   
&\leq &
  2 \left(   
  \sqrt{           \frac{1-\sqrt{\mu}}{1+\sqrt{\mu}}
     (1-\sqrt{\mu}) (1+\sqrt{\mu})                      }         \right)^k \sqrt{X_i^2 + Y_i^2}  \notag 
\\
& = &    2 \left(   
   1-\sqrt{\mu} \right)^k \sqrt{X_i^2 + Y_i^2} .
   \label{afsdfsafsa}       
\end{eqnarray}
\end{enumerate}

 Putting everything together, for all $k\geq 2$ we have
\begin{eqnarray*}
\|r_k\|_{\mB}^2 &=& \|\Exp{x_k - x_*}\|_{\mB}^2 = \| \mU^\top  \mB^{1/2} \Exp{x_k-x_*}\|^2 
= \| w_k \|^2 = \sum_{i=1}^n (w_k^i)^2 \\
&\overset{\eqref{afsdfsafsa}}{\leq} &
\sum_{i: \lambda_i = 0} \underbrace{(w_0^i)^2}_{=0}
+\sum_{i: \lambda_i = 1} \underbrace{(w_0^i)^2}_{=0} + \sum_{i: 0 < \lambda_i < 1} 4 (1-\sqrt{\mu})^{2k} (X_i^2 + Y_i^2)  \\ 
&=& 4 (1-\sqrt{\mu})^{2k} \sum_{i: 0< \lambda_i < 1}  (X_i^2 + Y_i^2).
\end{eqnarray*}
\end{proof}

An estimate of the constant $C$ is provided in Proposition~\ref{prop:C}  in the Appendix. We do not have a result on L2 convergence. We have tried to obtain an accelerated rate in the L2 sense, but were not successful.\footnote{In October 2017---slightly more than a year after the first version of this paper was circulated---an accelerated rate in the L2 sense was recently proved in \cite{RT-acc-2017}, but for a different method. The result was subsequently generalized in \cite{ASBFGS} in several ways, and applied to devising the first accelerated quasi-Newton matrix inversion formulas.}

\section{Conclusion and Extensions}

\subsection{Conclusion}

We have developed a generic scheme for  reformulating any linear system as a {\em stochastic problem}, which has several seemingly different but nevertheless equivalent  interpretations: stochastic optimization problem, stochastic linear system, stochastic fixed point problem, and probabilistic intersection problem.  While stochastic optimization is a broadly studied field with rich history, the concepts of stochastic linear system, stochastic fixed point problem  and probablistic intersection appear to be new.

We give sufficient, and necessary and sufficient conditions for the reformulation to be exact, i.e., for the solution set of the reformulation to exactly match the solution set of the linear system. To the best of our knowledge, this is is first systematic study of stochastic reformulations of linear systems. Further, we have developed three algorithms---basic, parallel and accelerated methods---to solve the stochastic reformulations. We have studied the convergence of expected iterates, L2 convergence, converge of iterate averages,  and convergence of $f$.  Our methods recover an array of existing randomized algorithms for solving linear systems in special cases,  including several variants of the randomized Kaczmarz method \cite{Strohmer2009}, randomized coordinate descent \cite{Leventhal:2008:RMLC}, and all the methods developed in \cite{SIMAX2015,SDA}.

\subsection{Extensions}

At the time of revising this manuscript for final publication\footnote{We have submitted the final revision in December 2019, which is more than 3 years since the first draft of this paper was written.} our work {\em has already been extended} in several ways, and many of these extensions were already published. We shall now briefly comment on some of them.


\paragraph{Convex feasibility} As shown by Necoara et al~\cite{SPM}, the results obtained in our paper for the parallel method generalize to the more general convex feasibility problem. Their paper is motivated by our work; indeed, Necoara et al~\cite{SPM} extend our stochastic reformulations to the convex feasibility setting. The generalization is tight. Their results resolve a major open problem in the convex feasibility literature on the efficiency of extrapolated parallel projection methods. 

\paragraph{Quasi-Newton literature} While we were unable to obtain an accelerated method for the quantity $\Exp{\|x_k-x_*\|^2_{\mB}}$, this problem was solved in \cite{RT-acc-2017}, and later further generalized by Gower et al~\cite{ASBFGS}.  As a by-product, they develop the first accelerated (and provably so) quasi-Newton matrix inversion formulae, despite half of a century of research on quasi-Newton methods. 

\paragraph{Variance reduction} A major open problem in the variance reduction literature for finite-sum optimization problems is shedding light on the mechanism behind variance reduction. Motivated by our work, Gower et el~\cite{JacSketch} develop a {\em controlled} stochastic reformulation of finite sum problems, which enables them to design the JacSketch method. In so doing, they prove that the variance reduction mechanism is nothing else than applying one step of the basic method (developed in our paper)  to a sequence of linear systems whose solution is  the Jacobian of a certain function evaluated at the latest iterate. As a consequence of their general analysis, they resolve a conjecture related to the optimal convergence rate of the celebrated SAGA method~\cite{SAGA} with importance sampling.

\paragraph{SGD} The results in our paper were instrumental in the development of the tightest known analysis of the SGD method by Gower et al~\cite{SGD-AS} -- the key method behind training supervised ML models. In  contrast with previous works, the analysis in \cite{SGD-AS} does not rely on any boundedness of the second moment assumptions, and is the first SGD analysis which recovers the rate of gradient descent as a special case. Their results rely on our paper in several ways: they generalize our basic method to the problem of  minimizing a quasi strongly convex function formed as the average of smooth functions. Their expected smoothness assumption---key to their results---was first introduced in our paper. In fact, in our paper it holds as an identity, while in general it holds as an inequality. Further, their development heavily relies on a new stochastic reformulation concept that applies to finite-sum problems; again, this concept was motivated by our work.

\paragraph{Stochastic spectral methods} The stochastic preconditioning idea introduced in our paper was later studied in Kovalev et al~\cite{SSCD}, who start with our general complexity results for the basic method, and proceed to develop optimized distributions $\cal D$ for the problem they consider. This leads to a new class of methods: stochastic {\em spectral} methods.

\paragraph{Average consensus} The methods in our work were further specialized (to specific linear systems and specific distributions $\cal D$) and adapted by Loizou and co-authors~\cite{NEW-PERSPECTIVE, agossip}  to  the average consensus problem -- a famous problem in the signal processing literature. They show that our methods lead to state-of-the-art average consensus methods.

\paragraph{Further extensions} Our results were also extended to include Polyak momentum~\cite{SMOMENTUM}, allow for inexact stochastic gradient computation~\cite{inexact_basic}, led to the development of the first coordinate descent methods that can handle any regularizer~\cite{SEGA}, and ultimately led to the first unification of stochastic gradient and coordinate descent methods~\cite{GJS}.

 Last but not least, we hope that our work provides a bridge across communities, including numerical linear algebra, stochastic optimization, machine learning, computational geometry, fixed point theory, applied mathematics and probability theory. We wish for our work to inspire further progress at the boundaries of these fields.

 \bibliographystyle{plain} 

 \bibliography{literature}
 
 \appendix

\section{Stochastic proximal point method}

As claimed in the introduction, here we show (see Theorem~\ref{thm:SPP} below) that the stochastic proximal point method  \eqref{eq:alg:SPP} is equivalent to stochastic gradient descent \eqref{eq:alg:SGD}.  First, we state a couple of lemmas, starting with the Sherman-Morrison-Woodbury matrix inversion formula \cite{SM1949,W1949}.

\begin{lemma}[Sherman-Morrison-Woodbury] \label{lem:098y9f098} Let $\mM \in \R^{n\times n}$, $\mC\in \R^{n\times q}$, $\mN \in \R^{q\times q}$ and $\mD \in \R^{q\times n}$, with $\mM$ and $\mN$ being invertible. Then 
\[(\mM + \mC \mN \mD)^{-1} = \mM^{-1} - \mM^{-1}\mC \left( \mN^{-1} + \mD \mM^{-1} \mC \right)^{-1}\mD \mM^{-1}.\]
\end{lemma}

The next result, Lemma~\ref{lem:ih8hjsuspo}, is trivially true if $\mM$ is positive definite. Indeed, in that case, $(\mM^\dagger)^{1/2} \mM  (\mM^\dagger)^{1/2} = \mI$, and the statement follows. However, in general, $(\mM^\dagger)^{1/2} \mM  (\mM^\dagger)^{1/2}$ is not equal to the identity; the lemma therefore says that the expression on the left hand side still behaves as if it was.

\begin{lemma}\label{lem:ih8hjsuspo} Let $\mM$ be a symmetric positive semidefinite matrix. Then for all $\mu>0$ we have the identity:
\begin{equation}\label{eq:iugs876tgug78s}(\mM^\dagger)^{1/2} \left(\mI + \frac{1}{\mu} (\mM^\dagger)^{1/2} \mM (\mM^\dagger)^{1/2}\right)^{-1}(\mM^\dagger)^{1/2} = \frac{\mu}{1+\mu} \mM^\dagger.\end{equation}
\end{lemma}
\begin{proof} Let $\mM = \mU \mD \mU^\top$ be the eigenvalue decomposition of $\mM$. Then $\mM^\dagger = \mU \mD^\dagger \mU^\top$, and it is easy to show that identity \eqref{eq:iugs876tgug78s} holds if it holds for $\mM$ being diagonal. If $\mM$ is diagonal, then matrices on both sides of  \eqref{eq:iugs876tgug78s}  are diagonal, which means we can compare the individual diagonal entries. It is easy to see that if $\mM_{ii}=0$, then the $i$th diagonal element of the matrices on both sides of \eqref{eq:iugs876tgug78s} is zero. If $\mM_{ii} >0$, then $i$th diagonal element of the matrix on the  left hand side of \eqref{eq:iugs876tgug78s} is
\[\mM_{ii}^{-1/2} \left(1 + \frac{1}{\mu}\right)^{-1}\mM_{ii}^{-1/2}=\mM_{ii}^{-1}\left(1 + \frac{1}{\mu}\right)^{-1} = \frac{\mu}{1+\mu} \mM_{ii}^{-1}.\]
\end{proof}

We are now ready to prove the equivalence result.

\begin{theorem} \label{thm:SPP} If $0<\omega \leq 1$, then Algorithms \eqref{eq:alg:SPP} and \eqref{eq:alg:SGD} are equivalent. That is, for every $x\in \R^n$, $\mu\geq 0$ and matrix $\mS$ with $m$ rows we have\footnote{Note that the identity trivially  holds for $\omega = 0$ if we understand the function on the right hand side in the limit sense: $\omega\to 0$ from the right. That is, $x = \arg \min_{z} \|z-x\|_{\mB}^2$.}
\[x - \omega \nabla f_{\mS}(x) = \arg\min_{z\in \R^n} f_{\mS}(z) + \frac{1-\omega}{ 2\omega}\|z-x\|_{\mB}^2 .\]
\end{theorem}
\begin{proof} The identity holds\footnote{In this case we interpret this identity as meaning that the vector on the left hand side is a minimizer of the function on the right hand side (as there may be multiple minimizers).} for $\omega=1$. This follows \eqref{eq:zero} in view of the fact that  $f_{\mS}$ is nonnegative. If $0<\omega <1$, then under the substitution $\mu = \frac{\omega-1}{\omega}$, the statement is equivalent to requiring that 
\begin{equation}\label{eq:98y98hfffs}x - \frac{1}{1+\mu} \nabla f_{\mS}(x) = \arg\min_{z\in \R^n} f_{\mS}(z) + \frac{\mu}{ 2}\|z-x\|_{\mB}^2 \end{equation}
holds for any $\mu>0$.

The minimizer of the stochastic fixed point iteration (right hand side of \eqref{eq:98y98hfffs}) can be computed by setting the gradient  to zero: $0= \mA^\top \mH (\mA z - b) + \mu \mB (z-x)$, whence
$z_* = (\mu \mB + \mA^\top \mH \mA )^{-1}(\mA^\top \mH b + \mu \mB x).$
In view of the formula for the stochastic gradient \eqref{eq:fder}, our goal is therefore to show that 
\begin{equation}\label{eq:kjhgs7hd99s}  x - \frac{1}{1+\mu}\mB^{-1}\mA^\top \mH (\mA x - b) = (\mu \mB + \mA^\top \mH \mA )^{-1}(\mA^\top \mH b + \mu \mB x).\end{equation} By comparing the terms  involving $x$ and those that do not in \eqref{eq:kjhgs7hd99s},  it is sufficient to show that
\begin{equation}\label{eq:EEE2} \mu (\mu \mB + \mA^\top \mH \mA )^{-1}\mB  = \mI - \frac{1}{1+\mu}\mB^{-1}\mA^\top \mH \mA,\end{equation}
and
\begin{equation}\label{eq:EEE1} (\mu \mB + \mA^\top \mH \mA )^{-1}\mA^\top \mH b = \frac{1}{1+\mu}\mB^{-1}\mA^\top \mH  b.\end{equation}

Let us now compute the inverse matrix in the expression defining $z_*$. First, we have
\begin{eqnarray}(\mu \mB + \mA^\top \mH \mA )^{-1} 
&=&\mB^{-1/2}\left(\mu \mI  +\mB^{-1/2}\mA^\top \mH \mA  \mB^{-1/2}\right)^{-1}\mB^{-1/2}.\label{eq:iugiu8938yhujs}
\end{eqnarray}
Let $\mK $ be the symmetric square root of the symmetric positive semidefinite matrix $(\mS^\top \mA \mB^{-1} \mA^\top \mS)^\dagger$. This means that we can write $\mH = \mS \mK^{2} \mS^\top$.
We now compute the inverse \eqref{eq:iugiu8938yhujs} by applying Lemma~\ref{lem:098y9f098} with $\mM =\mu \mI$, $\mC = \mB^{-1/2}\mA^\top \mS \mK$, $\mN = \mI$ (of appropriate size) and $\mD = \mC^\top$. First, the inverse of the matrix $\mu \mI + \mB^{-1/2}\mA^\top \mH \mA  \mB^{-1/2}$ is given by
\begin{eqnarray*}
 \frac{\mI}{\mu} - \frac{1}{\mu^2} \mB^{-1/2}\mA^\top \mS \mK \left(\mI + \frac{1}{\mu} \mK \mS^\top \mA  \mB^{-1}\mA^\top \mS \mK\right)^{-1} \mK   \mS ^\top \mA \mB^{-1/2}.
\end{eqnarray*}
In view of \eqref{eq:iugiu8938yhujs}, pre and post-multiplying both sides of the last identity by $\mB^{-1/2}$, and subsequently applying Lemma~\ref{lem:ih8hjsuspo}  with $\mM = \mS^\top \mA \mB^{-1}\mA^\top \mS$ yields
\begin{eqnarray*}
\left(\mu \mB + \mA^\top \mH \mA  \right)^{-1} &=&  \frac{\mB^{-1}}{\mu} - \frac{1}{\mu^2} \mB^{-1}\mA^\top \mS \mK \left(\mI + \frac{1}{\mu}\mK \mS^\top \mA \mB^{-1}\mA^\top \mS \mK\right)^{-1} \mK   \mS ^\top \mA \mB^{-1}\\
&\overset{\eqref{eq:iugs876tgug78s}}{=}& \frac{\mB^{-1}}{\mu} - \frac{1}{\mu^2}\mB^{-1} \mA^\top \mS \left(\frac{\mu \mK^2}{1+\mu}\right) \mS ^\top \mA \mB^{-1}\\
&=& \frac{\mB^{-1}}{\mu} - \frac{\mB^{-1} \mA^\top \mH \mA \mB^{-1}}{\mu(1+\mu)}.
\end{eqnarray*}
Given the above formula for the inverse,  \eqref{eq:EEE2} follows immediately. Identity \eqref{eq:EEE1} follows using the facts that $b = \mA x_*$ and $(\mB^{-1}\mZ)^2 = \mB^{-1}\mZ$, where $\mZ= \mA^\top \mH \mA$.
\end{proof}

\section{Smallest nonzero eigenvalue}

We are using the following inequality in the proof of Theorem~\ref{thm:alg1_complexity_second}.

\begin{lemma} \label{lem:WGWtight2} If Assumption~\ref{ass:main} holds, then for all $x\in \range{ \mB^{-1/2} \mA^\top }$ we have:
\begin{equation}\label{eq:89y9s8sss000}x^\top \mB^{-1/2} \Exp{\mZ} \mB^{-1/2} x \geq \lambda_{\min}^+(\mB^{-1/2} \Exp{\mZ} \mB^{-1/2} ) x^\top x\end{equation}
\end{lemma}

\begin{proof} For any matrix $\mM\in \R^{m\times n}$, the inequality
$x^\top \mM^\top \mM x \geq \lambda_{\min}^+(\mM^\top \mM) x^\top x$ holds for all $ x\in \range{\mM^\top}$. Applying this  with $\mM = (\Exp{\mZ})^{1/2} \mB^{-1/2}$, we see that \eqref{eq:89y9s8sss000} holds for all $x\in \range{\mB^{-1/2}(\Exp{\mZ})^{1/2} }$.
However,
 \begin{eqnarray*}\range{ \mB^{-1/2} (\Exp{\mZ})^{1/2}  } &=& \range{ \mB^{-1/2} (\Exp{\mZ})^{1/2}  (\mB^{-1/2} (\Exp{\mZ})^{1/2} )^\top }\\
 & =& \range{\mB^{-1/2} \Exp{\mZ} \mB^{-1/2}} = \range{\mB^{-1/2}\mA^\top},\end{eqnarray*}
 where the last identity follows by combining Assumption~\ref{ass:main} and  Theorem~\ref{thm:exactness}. 
\end{proof}

\section{Linear difference equations}

The proof of Theorem~\ref{thm:main-accelerated} uses the  following  standard  linear recurrence relations result \cite{FillmoreMarx1968,Elyadi2005}.

\begin{lemma}\label{lemma:DS}
Consider the following linear homogeneous recurrence relation of degree 2 with constant coefficients: 
$
 \xi_{k+1} = E \xi_k + F \xi_{k-1},
$
with $\xi_0, \xi_1\in \R$. 
\begin{itemize}
\item[(i)] $\xi_k\to 0$ if and only if both  roots of the characteristic polynomial, $r^2 - E r^2 - F$, lie strictly inside the unit complex circle.

\item[(ii)] Assume that $E^2+4F < 0$, i.e., that both roots are complex (the roots are $\alpha+i\beta$ and $\alpha-i\beta$, where $\alpha= E/2$ and $\beta = \sqrt{-E^2 -4F}/2$). Then there are (complex) constants $X, Y$, depending on the initial conditions $\xi_0,\xi_1$, such that
$
 \xi_{k} = 
 2 M^k \left( X \cos(\theta k) + Y \sin(\theta k)\right), 
$
where $M = \sqrt{\alpha^2+\beta^2}$, and $\theta$ is such that $\alpha = M\cos (\theta)$ and $\beta = M\sin( \theta)$.
\end{itemize}
\end{lemma}

\section{Estimation of constant $C$ from Theorem~\ref{thm:main-accelerated}}

\begin{proposition} \label{prop:C} Consider the setting of Theorem~\ref{thm:main-accelerated}(ii). In particular, we assume that $\omega=1$ and let $\mu=0.99 \lambda_{\min}^+$. Moreover, assume that $x_0=x_1$. Then for $k\geq 2$ we can choose 
\[C = (1+\sqrt{\mu})^2 \sum_{i \;:\; 0<\lambda_{i} < 1}  \frac{\lambda_i}{(\lambda_i-\mu) (1-\lambda_i)} \left(u_i^\top \mB^{1/2} (x_0-x_*)\right)^2 < 8 \kappa f(x_0),\]
 where $\kappa \eqdef \max_{0 < \lambda_i < 1} \frac{1}{(\lambda_i-\mu)(1-\lambda_i)} $. The maximum is attained either at $\lambda_i=\lambda_{\min}^+$ (the smallest nonzero eigenvalue), or for $\lambda_i$ being the largest eigenvalue that is not equal to 1.
\end{proposition}
\begin{proof}
Based on the proof of Theorem 5.3,  \begin{equation}\label{eq:nbu8fgvud98x}C = 4 \sum_{i: 0<\lambda_i <1} (X_i^2 + Y_i^2),\end{equation}
where  $X_i$ and $Y_i$ satisfy the equations (see \eqref{eq:nbifgisg89gu9f})
\begin{equation}\label{eq:n9f8f}w_k^i = 2 M_i^k ( X_i \cos (\theta_i k) + Y_i \sin (\theta_i k)), \quad k\geq 0,\end{equation}
where $w_k^i = u_i^\top \mB^{1/2} {\rm E}[x_k - x_*] $. Note that $\gamma = \frac{2}{1+\sqrt{\mu}} \in (1,2)$, $E_i=\gamma (1- \lambda_i) = \frac{2}{1+\sqrt{\mu}}(1-\lambda_i)$, $F_i=(1-\gamma)(1-\lambda_i) = -\frac{1-\sqrt{\mu}}{1+\sqrt{\mu}}(1-\lambda_i)<0$, $\alpha_i = \frac{E_i}{2} = M_i \cos \theta_i$, $  \beta_i = \frac{1}{2}\sqrt{-E_i^2-4F_i} = M_i \sin \theta_i,$ 
and $M_i = \sqrt{\alpha_i^2+\beta_i^2}=\sqrt{-F_i}  \in (0,1)$. Note that
\begin{equation}\label{eq:nb9gff8x98} 
\cos \theta_i = \frac{\alpha_i}{M_i}, \quad \text{and} \quad \sin \theta_i = \frac{\beta_i}{M_i}.\end{equation}

To identify $X_i$ and $Y_i$, it is enough to consider the first two equations in \eqref{eq:n9f8f}, i.e., $k=0,1$. For $k=0$ we get 
$w_0^i = 2 X_i$, from which we conclude that  \begin{equation} \label{eq:ni9f7g89hf0gd}X_i = \frac{1}{2} w_0^i = \frac{1}{2} u_i^\top \mB^{1/2} (x_0-x_*)\end{equation}
For $k=1$, and since $x_0=x_1$, we get  $w_0^i = w_1^i = 2M_i \left( X_i \cos \theta + Y_i \sin \theta \right) \overset{\eqref{eq:nb9gff8x98} }{=} 2 \left( X_i  \alpha_i + Y_i \beta_i \right),$
from which we get
\begin{equation} \label{eq:nbi8fg7dg9vd}
Y_i = \frac{1-\alpha_i}{\beta_i}X_i.
\end{equation}
By some arithmetic calculations, one obtains $(1-\alpha_i)^2 = \frac{(\lambda_i + \sqrt{\mu})^2}{(1+\sqrt{\mu})^2}$ and $\beta_i^2 =\frac{(\lambda_i-\mu)(1-\lambda_i)}{(1+\sqrt{\mu})^2}$, and subsequently
\[X_i^2+Y_i^2 \overset{\eqref{eq:nbi8fg7dg9vd}}{=} X_i^2\left( 1 + \frac{(1-\alpha_i)^2}{\beta_i^2}\right) = (1+\sqrt{\mu})^2  \frac{\lambda_i X_i^2}{  (\lambda_i-\mu)(1-\lambda_i)   }.\]
To obtain the formula for $C$, it remains to plug this into \eqref{eq:nbu8fgvud98x}.

The inequality $C<8 \kappa f(x_0)$ is obtained by estimating $\mu < 1$ (and hence $(1+\sqrt{\mu})^2 < 4$),  using the identity 
$f (x) =
\frac{1}{2} \sum_{i=1}^n \lambda_i \left(u_i^\top \mB^{1/2}(x -x_* )\right)^2$ (see \eqref{eq:f_spectral_decomp}), and applying the observation from the proof of Theorem~\ref{thm:main-accelerated} that $u_i^\top \mB^{1/2}(x_0 -x_* )=w_0^i=0$ for $i$ for which $\lambda_i=0$ or $\lambda_i=1$ (and $k\geq 2$).
\end{proof}

 \clearpage
\section{Notation glossary} \label{sec:notation_glossary}

\footnotesize

\begin{table}[!h]
\begin{center}
\begin{tabular}{|c|l|c|}
 \hline
 \multicolumn{3}{|c|}{{\bf The Basics} }\\
 \hline
$\mA, b$    & $m \times n$ matrix and $m\times 1$ vector & \\
                   & defining the system $\mA x =b$ & \\
$\cL$ &  $\{x\;:\; \mA x = b\}$ (solution set of the linear system) & \\
$\mB$    & $n \times n$ symmetric positive definite matrix & \\
$\langle x, y \rangle_{\mB}$ & $x^\top \mB y$ ($\mB$-inner product) & \\
$\|x\|_{\mB}$ & $\sqrt{\langle x, x \rangle_{\mB}}$ ($\mB$-norm) & \\
$\mM^{\dagger}$ & Moore-Penrose pseudoinverse of matrix $\mM$ & \\
$\mS$    & a random real matrix with $m$ rows  & \\
$\cD$    & distribution from which  $\mS$ is drawn ($\mS\sim \cD$) & \\ 
$\mH$ & $\mS (\mS^\top \mA \mB^{-1} \mA^\top \mS)^{\dagger} \mS^\top$ & \eqref{eq:H} \\
$\mZ$ & $\mA^\top \mH \mA$ & \eqref{eq:Z}\\
$\range{\mM}$ & range space of matrix $\mM$ & \\
$\kernel{\mM}$ & null space of matrix $\mM$ & \\
$\trace{\mM}$ & trace of matrix $\mM$ & \\
$\Prob(\cdot)$ & probability of an event & \\
$\Exp{\cdot}$ & expectation  & \\
 \hline
 \multicolumn{3}{|c|}{{\bf Projections}}\\
  \hline
$\Pi^{\mB}_{\cL}(x)$ & projection of $x$ onto $\cL$ in the $\mB$-norm & \eqref{eq:project_linear} \\
$\mM^{\dagger_\mB}$ & $\mB^{-1} \mM^\top (\mM \mB^{-1}\mM^\top)^\dagger$ ($\mB$-pseudoinverse of $\mM$) & \eqref{eq:B-pseudo} \\
$\mB^{-1}\mZ$ & projection matrix, in the $\mB$-norm, & \\
                          & onto $\range{\mB^{-1}\mA^\top \mS}$ & \eqref{eq:fder}  \\
 \hline
 \multicolumn{3}{|c|}{{\bf Optimization} }\\
 \hline
$\cX$ &  set of minimizers of $f$ & Thm~\ref{thm:X} \\ 
  $x_*$ & a point in $\cL$ & \\
$f_{\mS}$, $\nabla f_{\mS}$, $\nabla^2 f_{\mS}$ & stochastic function, its gradient and Hessian  & \eqref{eq:prodstoch}--\eqref{eq:v(x)}\\
$\cL_{\mS}$ & $\{x\;:\; \mS^\top \mA x = \mS^\top b\}$ (set of minimizers of $f_{\mS}$) & Lem~\ref{lem:all_sort_of_stuff_is_equal}    \\
$f$ & $\Exp{f_{\mS}}$  & \eqref{eq:min_f}, Lem~\ref{lem:f-various}\\
$\nabla f$ & gradient of $f$ w.r.t.\ the $\mB$-inner product & \\
$\nabla^2 f$ & $\mB^{-1}\Exp{\mZ}$ (Hessian of $f$ in the $\mB$-inner product) & \\
 \hline
 \multicolumn{3}{|c|}{{\bf Eigenvalues} }\\
 \hline
 $\mW$ & $\mB^{-1/2}\Exp{\mZ}\mB^{-1/2}$  & \\
             & (psd matrix with the same spectrum as $\nabla^2 f$) & \\
$\lambda_1,\dots,\lambda_n$ & eigenvalues of $\mW$  & \\
$\Lambda$ & $\Diag{\lambda_1,\dots,\lambda_n}$ (diag.\ matrix of eigenvalues) & \\
$\mU$ & $[u_1,\dots,u_n]$ (eigenvectors of $\mW$) & \\
$\mU \Lambda \mU^\top $ & eigenvalue decomposition of $\mW$ & \eqref{eq:eig_decomp} \\
$\lambda_{\max}, \lambda_{\min}^+$ & largest and smallest nonzero eigenvalues of $\mW$ & \\
$\condnum$ & $\lambda_{\max}/\lambda_{\min}^+$ (condition number of $\mW$) & \eqref{eq:condition_number}, \eqref{eq:kappa}\\
 \hline
 \multicolumn{3}{|c|}{{\bf Algorithms} }\\
 \hline
$\omega$ & relaxation parameter / stepsize & Alg~\ref{alg:alg1}--\ref{alg:alg3}  \\
$\mb$ & parallelism parameter & Alg~\ref{alg:alg2}\\
$\gamma$ & acceleration parameter &  Alg~\ref{alg:alg3}\\
\hline
\end{tabular}
\end{center}
\caption{Frequently used notation.}
\label{tbl:notation}
\end{table}

\end{document}

%% file: everything.tex
\usepackage{multirow}

 \usepackage{amsfonts}
 \usepackage{amsmath}
 \usepackage{amssymb} 

 \usepackage{xcolor, color, graphicx}
 
\newcommand{\R}{\mathbb{R}} 

\newcommand{\cD}{{\cal D}}

\newcommand{\cL}{{\cal L}}

\newcommand{\cO}{{\cal O}}

\newcommand{\cX}{{\cal X}}
\newcommand{\cY}{{\cal Y}}

\newcommand{\bS}{{\bf S}}

\newcommand{\mA}{{\bf A}}
\newcommand{\mB}{{\bf B}}
\newcommand{\mC}{{\bf C}}
\newcommand{\mD}{{\bf D}}

\newcommand{\mH}{{\bf H}}
\newcommand{\mI}{{\bf I}}

\newcommand{\mK}{{\bf K}}

\newcommand{\mM}{{\bf M}}
\newcommand{\mN}{{\bf N}}

\newcommand{\mP}{{\bf P}}

\newcommand{\mS}{{\bf S}}

\newcommand{\mU}{{\bf U}}

\newcommand{\mW}{{\bf W}}
\newcommand{\mX}{{\bf X}}

\newcommand{\mZ}{{\bf Z}}

\usepackage[colorinlistoftodos,bordercolor=orange,backgroundcolor=orange!20,linecolor=orange,textsize=scriptsize]{todonotes}

\newcommand{\eqdef}{\overset{\text{def}}{=}} 
\newcommand{\dotprod}[1]{\left< #1\right>} 
\newcommand{\norm}[1]{\lVert#1\rVert}      

\DeclareMathOperator{\Prob}{Prob}

\DeclareMathOperator{\prox}{prox}       

\newcommand{\Diag}[1]{\mathbf{Diag}\left( #1\right)}
\providecommand{\kernel}[1]{{\rm Null}\left( #1\right)}
\providecommand{\range}[1]{{\rm Range}\left( #1\right)}

\providecommand{\trace}[1]{{\rm Trace}\left( #1\right)}

\newcommand{\Exp}[1]{{\rm E}\left[#1\right] }    

\newcommand{\EE}[2]{{\rm E}_{#1}\left[#2\right] }

\newtheorem{assumption}{Assumption}
\newtheorem{example}{Example}
\newtheorem{remark}{Remark}